\documentclass[a4paper,11pt,pdf]{amsart}
\usepackage{enumerate, amsmath, amsfonts, amssymb, amsthm,  wasysym, graphics, graphicx, xcolor, frcursive,bbm}
\usepackage{etex}
\definecolor{darkblue}{rgb}{0.0,0,0.7} 
\newcommand{\darkblue}{\color{darkblue}} 
\definecolor{darkred}{rgb}{0.7,0,0} 
\usepackage{hyperref}
\usepackage[all]{xy}
\usepackage[T1]{fontenc}

\newcommand{\emp}[1]{\emph{\darkblue #1}} 
\newcommand{\ts}{\textsuperscript}

\usepackage{graphicx}                  
\usepackage{pstricks,pst-plot,pst-text,pst-tree,pst-eps,pst-fill,pst-node,pst-math}
\usepackage{setspace}

\newcommand\cB{{\mathcal B}}

\def\mb{\mathcal{B}_{2n}^{\mathrm{B,Mik}}}
\def\lt{\ell_{T}}
\def\ls{\ell_{S}}
\def\Bgnn{\mathcal{B}_{n+1}^{\mathrm{Mik}}}
\def\Bgn{\mathcal{B}_{n}^{\mathrm{Mik}}}
\def\Bpnn{\mathcal{B}_{n+1}^{\mathrm{per}}}
\def\bp{B(W)^{\mathrm{per}}}

\numberwithin{equation}{section}
\newtheorem{thm}[equation]{Theorem}
\newtheorem{theorem}[equation]{Theorem}
\newtheorem{prop}[equation]{Proposition}
\newtheorem{proposition}[equation]{Proposition}
\newtheorem{lem}[equation]{Lemma}

\newtheorem{corollary}[equation]{Corollary}
\newtheorem{conjecture}[equation]{Conjecture}

\theoremstyle{definition}
\newtheorem{defn}[equation]{Definition}
\newtheorem{definition}[equation]{Definition}
\newtheorem{rmq}[equation]{Remark}
\newtheorem{remark}[equation]{Remark}
\newtheorem{exple}[equation]{Example}
\newtheorem{problem}[equation]{Problem}

\def\resp{\mbox{\it resp}.\ }
\def\eg{\mbox{\it e.g.}}
\def\ie{\mbox{\it i.e}.}
\newcommand\inv{^{-1}}
\newcommand\lexp[2]{\kern\scriptspace\vphantom{#2}^{#1}\kern-\scriptspace#2}
\newcommand\ldiv{\mathbin{\preccurlyeq}}
\newcommand\DIV{\mathrm{Div}}
\newcommand\bc{{\boldsymbol c}}
\newcommand\bs{{\boldsymbol s}}
\newcommand\bt{{\boldsymbol t}}
\newcommand\bw{{\boldsymbol w}}
\newcommand\bx{{\boldsymbol x}}
\newcommand\bu{{\boldsymbol u}}
\newcommand\bv{{\boldsymbol v}}
\newcommand\by{{\boldsymbol y}}
\newcommand\bz{{\boldsymbol z}}
\newcommand\bD{{\boldsymbol D}}
\newcommand\bS{{\boldsymbol S}}
\newcommand\bT{{\boldsymbol T}}
\newcommand\bW{{\boldsymbol W}}
\newcommand\tl{{\mathrm{TL}_n}}

\newcommand\BZ{{\mathbb Z}}
\newcommand\BN{{\mathbb N}}

\DeclareMathOperator\rank{\mathrm{rank}}
\DeclareMathOperator\Dec{\mathrm{Dec}}

\title[Dual braid monoids, Mikado braids and positivity]{Dual braid monoids, Mikado braids and positivity in Hecke algebras}

\author{François Digne}
\address{LAMFA, Université de Picardie Jules Verne, 33 Rue Saint-Leu, 80039 Amiens Cedex 1, France.}
\email{digne@u-picardie.fr} 
\author{Thomas Gobet}
\address{TU Kaiserslautern, Fachbereich Mathematik, Postfach 3049, 67653 Kaiserslautern, Germany.}
\email{gobet@mathematik.uni-kl.de}

\begin{document}
\maketitle

\begin{abstract}
We study the rational permutation braids, that is the 
elements of an Artin-Tits group of spherical type which can be written $x^{-1} y$ where $x$ and $y$ are
prefixes of the Garside element of the braid monoid. We give a geometric characterization of these braids in
type $A_n$ and $B_n$ and then show that in spherical types different
from $D_n$ the simple elements of the dual braid monoid
(for arbitrary choice of Coxeter element) embedded in the braid group are
rational permutation braids (we conjecture this to hold also in type
$D_n$). This property implies positivity properties of the
polynomials arising in the linear expansion of their images in the Iwahori-Hecke algebra when expressed in the Kazhdan-Lusztig basis. In type $A_n$, it implies positivity properties of their images in the Temperley-Lieb algebra when expressed in the diagram basis.
\end{abstract}
\tableofcontents
\section{Introduction}

This paper is motivated by positivity properties arising when expanding simple dual braids in the canonical basis of the Iwahori-Hecke algebra. 

More precisely, a \emp{dual braid monoid} associated to a finite Coxeter
group and a choice of Coxeter element is a Garside monoid (\cite{DP}) whose group
of fractions is isomorphic
to the corresponding Artin-Tits group. The dual braid
monoids were introduced by Bessis \cite{Dual}, extending
definitions by Birman-Ko-Lee \cite{BKL} and Bessis-Michel and the first author
\cite{BDM}. As Garside monoids, they possess a finite set of \emp{simple elements}
which generates the whole monoid and forms a lattice under the left-divisibility order and they embed into their group of fractions.  Similar constructions have been made for some infinite Coxeter groups
and for complex reflection groups, but in this paper, unless explicitely
stated,  ``dual braid monoid''
will mean a dual braid monoid associated with
a finite Coxeter group.

In the framework of Artin-Tits groups, the standard example of a Garside
monoid is the braid or Artin-Tits monoid, also known as positive or classical
braid monoid. Its set of simple elements is the  canonical positive lift of
the Coxeter group in the Artin-Tits group. The embedding property in this
specific case was shown by Brieskorn-Saito (\cite{BS}) and by Deligne
(\cite{Del}). In the dual approach, the set of simple generators of the
Coxeter group is replaced in the spherical case by the whole set of
reflections, that is, by the conjugates of the simple generators. A
key tool in this approach is an action of the braid group on the set of reduced factorizations of a Coxeter element as a product of reflections, called the Hurwitz action. This action is known to be transitive (\cite{Dual}, \cite{BDSW}), providing an analogue in the dual approach of the Tits-Matsumoto property. 

The dual braid monoid is then an analogue to the classical braid monoid, but having as set of generators a copy of the set of reflections of the Coxeter group. In the classical setting, the lattice of simple elements is isomorphic to the lattice obtained by endowing the Coxeter group with the left weak order defined by the classical length function. By mimicking this approach in the dual setting, that is, by replacing the classical length function by the length function with respect to the whole set of reflections, one obtains an order called \emp{reflection} or \emp{absolute order} on the Coxeter group. The absolute order does not in general endow the Coxeter group with a lattice structure, but  by restricting such an order to the order ideal of a Coxeter element, one obtains a lattice (\cite{Dual}, \cite{BW}). The dual braid monoid is built using this property. Its lattice of simple elements is isomorphic to the order ideal of a given Coxeter element ordered by the restriction of the absolute order (\cite{Dual}).

There is a morphism from the Artin-Tits group to the group of invertible elements of the Iwahori-Hecke algebra attached to the corresponding Coxeter group. The images of the simple elements of the classical braid monoid through this morphism yield a basis of the Iwahori-Hecke algebra, the so-called \emp{standard basis}. It is therefore natural to ask about properties of the images of the simple elements of the dual Garside structure of an Artin-Tits group in the associated Iwahori-Hecke algebra.

In type $A_n$, it is known that their images in the Temperley-Lieb algebra
yield a basis (\cite{Z}, \cite{LL}, \cite{Vincenti}) such that the base change to the diagram basis is
triangular (\cite{Z}, \cite{GobTh}).

One of the main points of this paper consists of understanding how one can
express the simple dual braids in terms of the classical braid group
generators. Indeed, the isomorphism between the group of fractions of a dual
braid monoid and the Artin-Tits group is proven by a case-by-case analysis, and there is
no known general formula for the simple dual
braids in the braid group, even for the atoms. We recall the
construction of the classical and dual braid monoids in Sections \ref{sec:classicalbraidmonoid} and \ref{sec:dualbraidmonoid} where we also prove some new results on dual
Coxeter systems.  In particular we give a formula for the reflections in terms of the simple reflections (Proposition \ref{reflections in W}) which lifts well to the Artin-Tits group, allowing us to get a uniform formula for the dual atoms in
terms of the classical generators (Proposition \ref{prop:formula_dual_atoms}) and this with a case-free proof.

We then introduce in Section \ref{sec:good_permutation_braids} a finite set
of braids, which we call \emp{rational
permutation braids}. These braids already appear in a paper by
Dehornoy (\cite{Dehornoy}) in type $A_n$ and in unpublished work of Dyer (\cite{Dyernil}) for Artin-Tits groups attached to arbitrary (not necessarily finite) Coxeter groups. They may be
defined in spherical types as the braids of the form $\bx \by^{-1}$ where $\bx$ and $\by$ are
simple braids for the classical Garside structure. The images of such
braids in the Iwahori-Hecke algebra turn out to have positive Kazhdan-Lusztig
expansions; this is shown by Dyer-Lehrer (\cite{DL}) for finite Weyl groups using
perverse sheaves and by Dyer (\cite{D}) for arbitrary finite Coxeter groups using Soergel bimodules. These braids are closely related to the so-called \emp{mixed braid relations} introduced by Dyer (\cite{Dyernil}). 

We give uniform Garside-theoretic characterizations of the rational permutation braids in the spherical case in Proposition \ref{prop:gperm}, and then turn to a case-by-case analysis in types $A_n$ (Section \ref{sec:a}) and $B_n$ (Section \ref{sec:b}) to characterize them in terms of geometrical braids (Propositions \ref{prop:caract} and \ref{prop:caractb}). In these cases, they turn out to be what we call the \emp{Mikado braids}, that is, the braids where one can inductively
remove a strand which is above all the other strands. With this geometric
characterization, we can show that all the simple dual braids, for any choice
of Coxeter element, are rational permutation braids (Theorems
\ref{simplesmikadoa} and \ref{simplesmikadob})--the argument given here is
therefore topological. We also prove the same result for dihedral
groups and, by computer, for exceptional types (Theorem
\ref{simplesmikadoexc}); we conjecture the result to also hold in type $D_n$ (Conjecture \ref{conj:d}). 

It follows that the simple dual braids have a positive Kazhdan-Lusztig expansion in all types except possibly type $D_n$ (Theorem \ref{thm:positivity}). Positivity results in the Temperley-Lieb algebra (of type $A_n$) can also be derived (Theorem \ref{thm:positivitetl}).\\

\textbf{Acknowledgments}. We thank Patrick Dehornoy, Matthew Dyer and Jean
Michel for useful discussions. We also thank the anonymous referee for his careful reading of the manuscript and many interesting comments and remarks. 

\section{Classical braid monoid}\label{sec:classicalbraidmonoid}
Starting with a finite Coxeter System $(W,S)$,
one can define the associated Artin-Tits monoid or
\emp{classical braid monoid} $B^+(W)$ and for each choice of a Coxeter
element $c$ an associated \emp{dual braid monoid} $B^*_c(W)$
(\cite{Dual}). 

The classical braid monoid is defined by the following presentation:
$$
B^+(W)=\langle\bS\mid \underbrace{\bs_1\bs_2\ldots}_{m_{s_1,s_2}}=\underbrace
{\bs_2\bs_1\ldots}_{m_{s_1,s_2}} \text{ for } \bs_1,
\bs_2\in\bS\rangle^+$$
where $\bS$ is a set in bijection $\bs\mapsto s$ with $S$ and
$m_{s_1,s_2}$ is the order of $s_1s_2$ in $W$.

The monoid $B^+(W)$ is a \emp{Garside monoid} (see for instance \cite[I, 2.1]{DDGKM} for the definition), hence it satisfies the Ore conditions and  embeds in its group of fractions, the braid group $B(W)$ of $W$. This group has the same presentation as $B^+(W)$ but as a group:
$$B(W)=\langle\bS\mid \underbrace{\bs_1\bs_2\ldots}_{m_{s_1,s_2}}=\underbrace
{\bs_2\bs_1\ldots}_{m_{s_1,s_2}} \text{ for } \bs_1,
\bs_2\in\bS\rangle.$$

The group $W$ is a quotient of $B(W)$ by the normal subgroup generated by the
squares of the elements of $\bS$. We denote by $p:B(W)\rightarrow W$
this surjective morphism.
It has a set-wise section obtained by lifting $S$-shortest
expressions. We denote by $\bW$ the image of this section.
This set $\bW$ is the set of simples of the Garside monoid $B^+(W)$.
The Garside element of $B^+(W)$ is the lift $\Delta=\bw_0$ of the longest element
$w_0$ of $W$.

By definition of Garside monoids
the left- (\resp right-) divisibility gives a lattice structure to
$B^+(W)$ which restricts to a lattice structure on
$\bW$. Moreover by construction of $\bW$,
the map $p$ provides an isomorphism of lattices from $\bW$
to $W$ endowed with the Coxeter theoretic left- (\resp right-) weak order.

\section{Dual braid monoids}\label{sec:dualbraidmonoid} 

\subsection{Coxeter elements}
The dual braid monoids are defined using Coxeter elements. Before defining
these monoids we prove some properties of the Coxeter
elements that we will need.
\begin{definition} Let $(W,S)$ be a Coxeter system. An element $c\in W$ is a \emp{standard Coxeter
element of $(W,S)$} if it is a product of the elements of $S$ in some order. An element
$c\in W$ is a \emp{Coxeter element} if there exists $S'\subset T=\bigcup_{w\in W} wSw^{-1}$ such
that $(W,S')$ is a Coxeter system and $c$ is the product of the elements of $S'$ in some order,
that is, $c$ is a standard Coxeter element of $(W,S')$. The set $T$ is the set of \emp{reflections} of $(W,S)$.
\end{definition}
\begin{remark}
If $S'$ is as above then by \cite[Lemma 3.7]{muhl} one has $T=\bigcup_{w\in W} w S'
w^{-1}$. Moreover
if $W$ is finite and $S'$ is as above, then by \cite[Theorem 3.10]{muhl},
the Coxeter systems $(W,S)$ and $(W,S')$ are isomorphic. In particular, the type of $(W,S)$
depends only on $T$. If $(W,S)$ is infinite irreducible with no
infinite entry in its Coxeter matrix, a stronger property holds: any $S'$ such that
$(W,S')$ is a Coxeter system is conjugate to $S$ (\cite[Theorem 1]{FHM}).
The example of the dihedral group $I_2(5)$ with $S=\{s,t\}$ and
$S'=\{s,ststs\}$ shows that when $W$ is finite $S$ and $S'$ need not be
conjugate.
\end{remark}

\begin{definition} Let $(W,S)$ be a Coxeter system.
\begin{enumerate}

\item
We write $\ls$, \resp 
$\lt$, for the length function on $W$ with respect to $S$, \resp $T$.
\item We say that $x\in W$
\emp{divides} $y\in W$, written $x\ldiv_T y$, if $\lt(x\inv y)+\lt(x)=\lt(y)$. We write $\DIV(y):=\{x\in W\mid x\ldiv_T y\}$ and call the elements of $\DIV(y)$ the \emp{divisors} of $y$.
\end{enumerate}

Note that the definitions of $\ls$ and $\lt$ use the fact that
each one of the sets $S$ and $T$ positively generates $W$.
Note also that,
since $T$ is invariant under conjugation, the relation $\lt(x\inv y)+\lt(x)=\lt(y)$
is equivalent to  $\lt(yx\inv)+\lt(x)=\lt(y)$: there is no need to distinguish
between left- and right-divisibility.
\end{definition}
The following proposition shows how to
recover all reflections from a Coxeter element when $W$ is finite.
Recall that if $c$ is a Coxeter element in a Coxeter system of rank $n$,
one has $\lt(c)=n$ (see \eg, \cite[Lemma 1.2]{BDSW}). 
\begin{proposition}\label{reflections in W}
Let $(W,S)$ be a finite Coxeter system with $S=\{s_1,\ldots,s_n\}$. Let $c=s_1s_2\cdots s_n$. Then the set of reflections of $W$ is given by
$$\{c^ks_1s_2\ldots s_is_{i-1}\ldots s_1 c^{-k}\mid k\in\BN,\; 1\leq i\leq
n\}.$$
\end{proposition}

\begin{proof}
Let $T_c$ be the set of the statement and $T$ be the set of reflections of
$(W,S)$. We
have $T_c\subset T$. If $c$ is such that $\ls(c^i)=i\ls(c)$ for
any $i\leq |T|/h$, where $h=2|T|/n$ is the Coxeter number (the order of $c$),
then $T_c$ is equal to $T$
(this is a consequence of \eg, the proof of \cite[3.9]{broue-michel};
see also \cite[Chapter V, \S 6 exercise 2]{bourbaki} in the case where
$c$ is bipartite), whence the result holds in that case.
We now remark that if $c'=s_{i+1}\ldots
s_ns_1\ldots s_i$, with $1\leq i\leq n-1$ then $T_{c'}$ contains
$(s_i\ldots s_2s_1)T_c(s_i\ldots s_2s_1)^{-1}$ so has cardinality at least
$|T_c|$, whence the result holds: indeed, since $W$ is finite,
all standard Coxeter elements of $(W,S)$ are cyclically
conjugate (that is by a sequence of conjugations like the one deriving
$c'$ from $c$ above): see \eg, \cite[Theorem 3.1.4]{GP}.
\end{proof}

We take from \cite[Section 1]{BDSW} the second item of the following definition.
\begin{defn}\label{parabolic}
\begin{enumerate}
\item A \emp{parabolic subgroup} $W'$ of a Coxeter system $(W,S)$ is
a subgroup generated by a conjugate $S'$ of a subset of $S$.
We shall say that $(W',S')$ is a \emp{parabolic subsystem} of $(W,S)$.

\item Let $(W,S)$ be a Coxeter system with set of reflections $T$. Let $x\in W$.  Then $x$ is a \emp{parabolic Coxeter element} in $W$ if there exists a subset $S'=\{s'_1,\dots, s'_n\}\subset T$ such that $x=s'_1\cdots s'_m$ for some $m\leq n$ and $(W, S')$ is a Coxeter system. 
\end{enumerate}
\end{defn}

\begin{corollary}\label{coxeter parabolique}
Let $(W,S)$ be a finite Coxeter system and let $x\in W$. The following are equivalent:
\begin{enumerate}
\item There exists a Coxeter element $c\in W$ such that $x\ldiv_T c$. 
\item The element $x$ is a parabolic Coxeter element.
\end{enumerate}
\end{corollary}
\begin{proof}
Assume that $x\ldiv_T c$ for some Coxeter element $c$. Let
$S'=\{s_1,\ldots,s_n\}$ be a simple Coxeter system of $(W,S)$ such that
$c=s_1\ldots s_n$ and let $T$ be the set of reflections of $W$.
We argue by reverse induction on $\lt(x)$. First assume that
$\lt(x)=n-1$, so that there exists
a decomposition of $c$ into a product of $n$ reflections
$c=t_1\ldots t_n$ such that $x=t_2\ldots t_n$.
By the above proposition we have $t_1=\lexp{c^ks_1\ldots s_i}s_{i+1}$ for
some $k\in\BN$ and some $1\leq i\leq n-1$, where for two elements $a$ and $b$
in a group, we write $\lexp ab$ for $aba\inv$.
Let $y=s_i\ldots s_1c^{-k}$;
we have $\lexp y c=s_{i+1}\ldots s_ns_1\ldots s_i$ and $\lexp yt_1=s_{i+1}$.
Then $\lexp y x=s_{i+2}\ldots s_ns_1\ldots s_i$
is a standard Coxeter element in the standard parabolic subgroup $W_1$
generated by $S_1=\{s_1,\ldots, s_i,s_{i+2},\ldots s_n\})$ so that $x$ is
a standard Coxeter element in the parabolic subsystem $\lexp{y\inv}(W_1,S_1)$. Since a conjugate of a simple system is again a simple system, we deduce that $x$ is a parabolic Coxeter element.
Now the reflections of a parabolic
subgroup are precisely the reflections of $W$ which lie in this subgroup, so
that the induction can go on: a divisor of $x$ in $W$ is in the parabolic
subsystem of which $x$ is a standard Coxeter element. 
The converse is immediate.
\end{proof}
\begin{rmq} In his case, that is for finite Coxeter groups,
Bessis gives \cite[Section 1.4]{Dual} an alternative definition of parabolic
Coxeter elements and shows \cite[Lemma 1.4.3]{Dual} that an element $x$ of
$W$ is a parabolic Coxeter element in his sense if and only if there exists a
bipartite Coxeter element $c$ such that $x\ldiv_T c$. Since in finite
Coxeter groups any Coxeter element is conjugate to a bipartite one, one can
drop the word bipartite from this definition. This last fact and the above lemma show that Bessis approach and the one of \cite{BDSW},
that is of Definition \ref{parabolic} (2), actually lead to equivalent
definitions for finite Coxeter groups; this equivalence is not obvious if we compare the definitions from Bessis and \cite{BDSW}. 
 \end{rmq}

\subsection{Dual braid monoids}\label{sub:dbm}
The dual braid monoids are defined as follows: let $(W,S)$ be a Coxeter
system and let $T$ denote the set of
reflections of $(W,S)$.
Let $c$ be a standard Coxeter element.
The dual braid monoid $B_c^*$ associated to $c$ is defined by generators and relations as
follows. 
Let $\bD_c$ be a set in one-to-one correspondence with the
set $\DIV(c)$ of divisors of $c$. Then 
$B^*_c(W)$ is generated by $\bD_c$ with only relations $\bx\by=\bz$
for $\bx$, $\by$, $\bz$ in $\bD_c$ such that $xy=z\in\DIV(c)$ and
$\lt(z)=\lt(x)+\lt(y)$, where $x, y, z\in\DIV(c)$ are the elements
corresponding to $\bx,\by,\bz$ respectively.
The canonical bijection $\bD_c\rightarrow \DIV(c)$ extends to a
(surjective) morphism of monoids from $B^*_c(W)$ to $W$.  The above definition is valid for any Coxeter system but in the sequel
we study only dual braid monoids associated to finite Coxeter groups. 

In the spherical case the monoid $B^*_c(W)$ is a Garside monoid
with $\bD_c$ as its set of simples. In this case the monoid $B^*_c(W)$ has
a presentation with a generating set smaller
than $\bD_c$: first it can be shown that any reflection divides $c$ (see \cite[Lemma 1.3.3]{Dual}); if we denote by $\bT_c$ the lift of $T$ in $\bD_c\subset
B^*_c(W)$ we have
\begin{proposition}[\cite{Dual}, Theorem 2.1.4]\label{presentation of B*}
In the spherical case the dual braid monoid has the following presentation:
$$
B^*_c(W)=\langle\bT_c\mid \bt_1\bt_2=\bt_2\bt_3 \text{ for } \bt_1,\bt_2,\bt_3
\text{ in }\bT_c \text{ with } t_1t_2=t_2t_3\ldiv_T c\rangle^+.
$$
\end{proposition}
The relations in the above presentation
of $B^*_c(W)$ are called the \emp{dual braid relations}.  

\begin{proof}
We recall the proof. To this end, we need to recall a few facts on the Hurwitz action on sets of $T$-reduced decompositions of elements.

\subsection{Hurwitz action on reduced decompositions} 
We write $\mathcal{B}_n$ for the Artin braid group on $n$ strands, that is, the Artin-Tits group $B(W)$ where $W$ is of type $A_{n-1}$.

\begin{definition}[Hurwitz action]
Let $G$ be a group and $g$ be in $G$; the braid group $\mathcal{B}_n$ acts on the set
of $n$-tuples of elements of $G$ whose product is equal to $g$ as follows:
if $\sigma_1,\ldots,\sigma_{n-1}$ are the generators of the braid group $\mathcal{B}_n$,
the \emp{Hurwitz action} of
$\sigma_i$ maps a sequence $(t_1,\ldots,t_n)$  such that $g=t_1t_2\ldots t_n$
to the sequence
$(t_1,\ldots,t_{i-1},t_it_{i+1}t_i\inv,t_i,t_{i+2},\ldots,t_n)$.
\end{definition}
When considering the Hurwitz action, we will always denote the generators of $\mathcal{B}_n$ by
$\sigma_1,\dots, \sigma_{n-1}$. Otherwise we write the generators of $B(W)$ for $W$ of type
$A_{n-1}$ as $\bs_1,\dots,\bs_{n-1}$. In the case of reduced factorizations of a Coxeter element into reflections we have:

\begin{theorem}[\cite{BDSW}, Theorem 1.3]\label{hurwitz}
Let $(W,S)$ be a (not necessarily finite) Coxeter system with $T$ its set of reflections. 
If $x$ is a parabolic Coxeter element in $W$, then the Hurwitz action is transitive
on the set of decompositions of $x$ into a product of $\lt(x)$ reflections.
\end{theorem}
Note that for finite Coxeter groups, the assumption that $x$ is a parabolic Coxeter element can be replaced by $x\ldiv_T c$
for some Coxeter element $c$ of $(W,S)$ thanks to Corollary
\ref{coxeter parabolique}.

Let $c$ be a standard Coxeter element in a Coxeter system $(W,S)$ of rank $n$.
By definition of $B^*_c$, any decomposition $c=t_1\ldots t_n$ into a product of 
$n$ reflections is
lifted in $B^*_c(W)$ to $\bc=\bt_1\ldots \bt_n$ where $\bc$ is the lift of $c$
in $\bD_c$. Hence if $x$ divides $c$, writing $x=t_1\ldots t_{\lt(x)}$ with
$t_i\in T$ and $c=t_1\ldots t_n$, we see that the lift of $x$ is a 
product of $\lt(x)$ elements of $\bT_c$, so that $\bT_c$ generates $B^*_c(W)$.
Moreover by Theorem \ref{hurwitz} and Corollary \ref{coxeter parabolique}
one can pass from any decomposition of $\bx\in\bD_c$ to any other one by dual
braid relations. This gives Proposition \ref{presentation of B*}.
\end{proof}
\subsection{Groups of fractions of dual braid monoids are Artin-Tits groups}
The following is proved in \cite{Dual}:
\begin{proposition}\label{frac(B*)=B}
The group of fractions of $B^*_c(W)$ is isomorphic 
to B(W) and the restriction of the projection $p:B(W)\rightarrow W$ is the
canonical surjection of the dual braid monoid onto $W$.
\end{proposition}

We explain how this isomorphism is defined and sketch a proof that it is an
isomorphism. 
\begin{proof}[Sketch of proof]
Let us enumerate the elements $s_1,\ldots,s_n$ of $S$
in such a way that $c=s_1\ldots s_n$. 
We have $S\subset T$ and an easy computation using the dual braid relations
shows that the lift $\bS'$ of $S$ in $\bT_c\subset B_c^*(W)$ satisfies the braid relations. 
This gives a morphism of monoids $B^+(W)\rightarrow B^*_c(W)$ mapping $\bs_i$ to
$\bs'_i$ where $\bs'_i$ is the lift of $s_i$ in $\bS'$. This extends to
a group morphism from $B(W)$ to the group of fractions of $B^*_c(W)$.

Since $c=s_1\ldots s_n$, whence $\bc=\bs'_1\ldots\bs'_n$,
Theorem \ref{hurwitz} implies that all elements of $\bT_c$ can be obtained from
$\bS'$ using dual braid relations, so that $\bS'$ generates the group of
fractions of $B^*_c(W)$, hence
the above morphism from $B(W)$ to this group is surjective.

To show that it is injective one first considers the particular case of a
Coxeter element $c$ which is bipartite, that is, such that either for
$m=\lfloor n/2\rfloor$ or for $m=\lfloor (n+1)/2\rfloor$  the elements
$s_1$, $s_2$, \dots, $s_m$ pairwise
commute and so do $s_{m+1}$, \dots, $s_n$. 
Bessis \cite{Dual} proves that
in this case the set $\cup_{k\in\BZ}\bc^k\bS\bc^{-k}$ is a lift in
$B(W)$ of $T$ and that its elements satisfy the dual braid relations (this last
fact by a case by case analysis, see \cite[Fact 2.2.4]{Dual}). Hence we get a
morphism in the other direction such that the composition from $B(W)$ to
itself is the identity (since it is the identity on $\bS$).

To get the result for an arbitrary standard Coxeter element, one can use the fact that
the lifts in $B^+(W)$ of all Coxeter elements of a given Coxeter system
are conjugate in $B(W)$ (see
\cite[Chapter V, \S 6 Lemme 1]{bourbaki}) and
that the Hurwitz action commutes with conjugation.
\end{proof}

We will now identify $B_c^*(W)$ with its image in $B(W)$ and $\bS'$ with $\bS$. We
then see $\bT_c$ as a subset of $B(W)$. Using this identification, we note
that, by the above proof, the lift of the Coxeter element $c=s_1\ldots s_n$ in $B_c^*(W)$ is the same
as its lift to the classical braid monoid, that is
$\bc=\bs_1\ldots\bs_n\in\bW$. Theorem \ref{hurwitz} implies:
\begin{corollary}\label{hurwitz in B} The 
set of decompositions of $\bc$ into a product of $\rank(W)$
elements of $\bT_c$ in $B(W)$ 
is a single orbit under the Hurwitz action and the
restriction of $\pi:B(W)\rightarrow W$
to $\bT_c$ induces a bijection from this Hurwitz orbit to  
the set of decompositions of $c$ into $\rank(W)$ reflections.
\end{corollary}
\noindent We recall the proof for sake of completeness.
\begin{proof} Let $\Dec_n(c)$ (\resp $\Dec_n(\bc)$) be the set
of decomposition of $c$ (\resp $\bc$) into $n$ reflections (\resp into $n$
elements of $\bT_c$) where $n=\rank(W)$.
The morphism $\pi$ bijectively maps $\bT_c$ to $T$, hence maps $\Dec_n(\bc)$ to
$\Dec_n(c)$.
By definition of $B_c^*(W)$ a decomposition $(t_1,\ldots ,t_n)\in\Dec_n(c)$
is lifted in $B(W)$ to $(\bt_1,\ldots, \bt_n)\in\Dec_n(\bc)$.
Hence $\pi$ induces a bijection from $\Dec_n(\bc)$ to $\Dec_n(c)$.

By the dual braid relations, the Hurwitz action in $B(W)$
preserves $\Dec_n(\bc)$.
Since $\pi$ is compatible with the Hurwitz action and
by Theorem \ref{hurwitz} $\Dec_n(c)$ is a single orbit, we 
get the result.
\end{proof}

A question is then to understand which elements of $B(W)$ correspond to the
subset $\bD_c$ of $B_c^*(W)$ through the embedding of $B_c^*(W)$ into $B(W)$.
This may be of interest for both computational reasons and for a possible
case-free proof of Proposition \ref{frac(B*)=B}.

We will address this question in Sections \ref{sec:a} and \ref{sec:b} when $W$ is of type $A_n$ or $B_n$. In the case of reflections the following subsection gives an answer.
\subsection{A formula for dual atoms}

The following proposition is the analog
for the braid monoid of Proposition \ref{reflections in W}
and the proof is parallel. 

\begin{proposition}\label{prop:formula_dual_atoms}
Let $c=s_1\ldots s_n$ be a standard Coxeter element,
then, through the embedding of $B_c^*(W)$ into $B(W)$, we have
$$\bT_c=\{
\bs_1\bs_2\ldots\bs_i\bs_{i+1}\bs_i\inv\bs_{i-1}\inv\ldots\bs_1\inv,
\mid 0\leq i< 2|T|\},$$ where the index $i$ in $\bs_i$ is taken modulo $n$.
\end{proposition}
\begin{proof}
Let $\bT'_c$ be the set of elements
$\bs_1\bs_2\ldots\bs_i\bs_{i+1}\bs_i\inv\bs_{i-1}\inv\ldots\bs_1\inv$ with
$i\geq 0$. The elements of $\bT'_c$ are components of the elements of the Hurwitz
orbit of $(\bs_1,\ldots,\bs_n)$: this is clear if $i\leq n$ and the elements with $i>n$ are
conjugates from the previous ones by a power of $c$.
Hence by Corollary \ref{hurwitz in B}  we have 
$\bT'_c\subset \bT_c$. Since by definition $|\bT_c|=|T|$, for proving
$\bT_c=\bT'_c$ it is sufficient to prove $|\bT'_c|\geq |T|$.
If $c$ is a Coxeter element such that $\ls(c^i)=i\ls(c)$ for any $i\leq |T|/h$,
where $h=2|T|/n$ is the Coxeter number (the order of $c$),
\eg, if $c$ is bipartite,
then the image in $W$ of $\bT'_c$ is the full set $T$
of reflections of $W$ as seen in the proof of Proposition
\ref{reflections in W}, so that the result holds in that case.
We now remark that if $\bc=\bs_1\ldots\bs_n$ and $\bc'=\bs_i\ldots
\bs_n\bs_1\ldots\bs_{i-1}$ then $\bT'_{c'}$ contains
$(\bs_i\ldots\bs_n)\bT'_c(\bs_i\ldots\bs_n)\inv$ so has cardinality at least
$|\bT'_c|$, whence the result holds, using the fact that all standard Coxeter elements of a given
Coxeter system are cyclically
conjugate.
By \cite[Chapter V, \S 6 exercise 2]{bourbaki}  if $\bc$ is bipartite, we have
$\bc^h=\bw_0^2$, where 
$\bw_0$ is the lift in $B^+(W)$ of the longest element of $W$.
Since $\bw_0^2$ is central in $B(W)$,
one gets the same set $\bT'_c$ of elements
$\bs_1\bs_2\ldots\bs_i\bs_{i+1}\bs_i\inv\bs_{i-1}\inv\ldots\bs_1\inv$ for
 $0\le i< 2|T|$ as for $i\in \BN$.
\end{proof}

\section{Rational permutation braids}\label{sec:good_permutation_braids}

\subsection{Square-free braids}
Let $(W,S)$ be a finite Coxeter system. Recall that we denote by $p:B(W)\rightarrow W$ the canonical surjection. 

\begin{defn}
An element $\beta\in B(W)$ is a \emp{square-free braid} if it can be
represented by a braid word $$\bs_{i_1}^{\varepsilon_1}\cdots \bs_{i_k}^{\varepsilon_k}$$
with $\varepsilon_j\in\{-1,1\}$ such that $k=\ell_S(p(\beta))$. We denote by
$\bp\subset B(W)$ the set of square-free braids.
\end{defn}

One has $\bW=B^+(W)\cap\bp$ which is also the set of prefixes of the Garside
element $\Delta$ (see \cite[Example 1]{DP}). It follows from the definition
that any square-free braid is obtained by replacing the various $s_{i_j}$ for $j=1,\dots, k$ in a reduced $S$-decomposition $s_{i_1}\cdots s_{i_k}$ of an element $w\in W$ by $\bs_{i_j}^{\pm 1}$. 

\subsection{Rational permutation braids}

As any group of fractions of a cancellative monoid $B(W)$ is endowed with a left-invariant
partial order
that we denote by $\ldiv$ defined by $u\ldiv v$ if $u^{-1}v\in B^+(W)$
which extends the left-divisibility relation on $B^+(W)$.
\begin{defn}
An element $\beta\in B(W)$ is a \emp{rational permutation braid} if it lies in the interval $[\Delta^{-1},\Delta]$ for the partial order $\ldiv$.
\end{defn}
The following proposition can be seen as a particular case of
\cite[Proposition 5.3]{Dehornoy}.
It explains the terminology as the elements of $\bW$ are usually called
permutation braids.
\begin{proposition}\label{prop:gperm}
Let $\beta\in B(W)$. The following are equivalent:
\begin{enumerate}
\item The braid $\beta$ is a rational permutation braid,
\item There exist $\bx$, $\by\in \bW$ such that $\beta=\bx^{-1}\by$,
\item There exist $\bx$, $\by\in \bW$ such that $\beta=\bx\by^{-1}$.
\end{enumerate}
\end{proposition}
\begin{proof}
We prove the equivalence of $(1)$ and $(2)$. The equivalence of $(1)$ and
$(3)$ is similar.

Since $B^+(W)$ satisfies the Ore conditions, any element of $B(W)$ can be
uniquely written as $x\inv y$ with $x$ and $y$ in $B^+(W)$ having no common
left-divisor. We have $\Delta\inv\ldiv x\inv y\ldiv\Delta$ if and only if 
$x\inv y=\Delta\inv a=\Delta b\inv$ with $a,b$ in $B^+(W)$. So $(1)$ implies that
$a$ is a divisor of $\Delta^2$ which by the general properties of Garside
monoids (see \cite[V, Proposition 3.24]{DDGKM}) means that $a=a_1a_2$ with $a_1,a_2\in \bW$,
whence $x\inv y=(a_1\inv\Delta)\inv a_2$. The Ore conditions imply then that
$x$ and $y$ are right-divisors of respectively $a_1\inv\Delta$ and $a_2$ in $B^+(W)$, so that $x$ and $y$ are in $\bW$.
Conversely, if $x$ and $y$ are in $\bW$, that is divide $\Delta$,
then $x\inv y\ldiv x\inv\Delta\ldiv\Delta$ and
$\Delta\inv\ldiv\Delta\inv y\ldiv x\inv y$.
\end{proof}

\noindent We thank Matthew Dyer for having 
pointed to us the property stated in the following lemma which is a particular case of \cite[Section 9.4]{Dyernil}. We reprove it here in our particular case.

\begin{lem}\label{good=>permutation}
If $\beta$ is a rational permutation braid and if $s_1s_2\cdots s_k$ is any
reduced expression of its image in $W$, then for $i=1,\dots, k$ there exists
$\varepsilon_i=\pm 1$ such that
$\beta=\bs_1^{\varepsilon_1}\bs_2^{\varepsilon_2}\cdots
\bs_k^{\varepsilon_k}$. In particular, any rational permutation braid is a
square-free braid.
\end{lem}
\begin{proof}
We prove by induction on $k$ that if $s_1s_2\ldots s_k$ is a reduced decomposition of an element of $W$, then for $\by\in\bW$ the element 
$\bs_1^ {\varepsilon_1}\ldots\bs_k^{\varepsilon_k}\by$, where $\varepsilon_i=1$ if
$\ls(s_is_{i+1}\ldots s_ky)=\ls(s_{i+1}\ldots s_ky)+1$ and $\varepsilon_i=-1$ otherwise is in $\bW$.
This concludes the proof, since if $\beta=\bx\by\inv$ is a rational permutation braid with $\bx$ and $\by$ in $\bW$ and if $s_1\ldots
s_k$ is a reduced decomposition of the image of $\beta$ in $W$, then $\bs_1^ {\varepsilon_1}\ldots\bs_k^{\varepsilon_k}\by$
is in $\bW$ and has same image as $\bx$, hence is equal to $\bx$, so that
$\beta=\bx\by\inv=\bs_1^ {\varepsilon_1}\ldots\bs_k^{\varepsilon_k}$ is a
square-free braid.

By induction hypothesis $\bx'=\bs_2^ {\varepsilon_1}\ldots\bs_k^{\varepsilon_k}\by$ is in $\bW$. If $\varepsilon_1=1$ then
$\bs_1\bx'$ is in $\bW$. If $\varepsilon_1=-1$ then, by the exchange lemma, $x'$ has a reduced expression of the form
$s_1s'_2\ldots s'_m$  so that $\bs_1^{\varepsilon_1}\bx'=\bs'_2\ldots\bs'_m$ is again in $\bW$.
\end{proof}
\begin{exple}\label{exple:typeamixed}
Let $W$ be of type $A_2$, with simple generating reflections $s_1$ and
$s_2$;
the element $w:=s_1s_2s_1$ in $W$ has exactly two reduced expressions
$s_1s_2s_1=s_2s_1s_2$. There are six square-free braids having $w$ as image.  For example,
$\beta=\bs_1^{-1}\bs_2\bs_1$. Here the reduced expression $s_1s_2s_1$ has been lifted to the braid word $\bs_1^{-1}\bs_2\bs_1$;  Lemma \ref{good=>permutation} says that we can lift any reduced expression, in particular we can also lift the reduced expression $s_2s_1s_2$ of $w$ in which case one has $\beta=\bs_2\bs_1\bs_2^{-1}$.
\end{exple}
\begin{remark}
Note that the set of rational permutation braids is the set of braids where
"one can apply (mixed) braid relations in reduced words as in the Coxeter
group". Let us be more precise: if $\beta$ is a rational permutation braid, then one can pass from any shortest expression of $\beta$ with respect to
$\bS\cup\bS\inv$ to any other one by mixed braid relations as defined in
\cite{Dyernil}, that is, braid
relations possibly involving inverses. Moreover, to any reduced expression of the image of $\beta$ corresponds a shortest braid word for $\beta$. Precisely given any two reduced
expressions of the image of a braid $\beta$ in $W$ which differ by a braid
relation, then by Lemma \ref{good=>permutation} both reduced expressions can
be lifted to reduced braid words for $\beta$ which differ by a mixed braid
relation. For example the relation $\bs_1^{-1}\bs_2\bs_1=\bs_2\bs_1\bs_2^{-1}$ in Example \ref{exple:typeamixed} is a mixed braid relation. 
\end{remark}

\section{Mikado braids of type $A_n$}\label{sec:a}
In this section $(W,S)$ is a Coxeter system of type $A_n$. The
group $W$ is identified with the symmetric group $\mathfrak{S}_{n+1}$ on
$\{1,2\ldots,n+1\}$ and $S$ with the set of simple transpositions, writing $s_i:=(i,i+1)$. The Artin-Tits group $B(W)$ is identified with $\mathcal{B}_{n+1}$.
\subsection{Square-free braids}
We give some geometrical properties of square-free braids of type $A_n$.
Firstly, notice that $\beta\in\mathcal{B}_{n+1}$ is a square-free braid if and only if there is a braid diagram for $\beta$ where any two strands cross at most once. We call a braid diagram \emp{reduced} if it has the minimal number of crossings. 
\begin{lem}\label{lem:remove}
Removing any strand in a diagram of a square-free braid $\beta\in
\mathcal{B}_{n+1}$ gives a diagram for a square-free braid $\beta'\in \mathcal{B}_n$.
\end{lem}
\begin{proof} Let $D$ be a diagram for $\beta$ and $D'$ be the diagram
obtained by removing a given strand. 
Since $\beta$ is a square-free braid, there is an isotopy which allows one to pass from $D$ to a diagram $\widetilde{D}$ where any two strands cross at most once. If we forget the strand we want to remove, our isotopy deforms $D'$ in a diagram $\widetilde{D}'$ obtained from $\widetilde{D}$ by removing the strand, and in $\widetilde{D}'$ any two strands among the remaining $n$ strands cross each other at most once since they were crossing each other at  most once in $\widetilde{D}$. 
\end{proof}

\subsection{Mikado braids}\label{sec:mikbr}
\begin{defn}
A strand of a given reduced diagram $D$ of a braid $\beta\in\mathcal{B}_{n+1}$ which is over the other strands it crosses is called \emp{good}.
\end{defn}
\begin{rmq} A strand is good in a reduced diagram for a braid $\beta\in \mathcal{B}_{n+1}^{\mathrm{per}}$ if and only if it is good in any reduced diagram for $\beta$: indeed, we first claim that the number of crossing between any two strands is constant on the set of reduced diagrams of a square-free braid. To see this, notice that this property holds for the reduced diagrams of the permutation obtained as image of $\beta$ in the symmetric group, and that the reduced diagrams for the various square-free braids having this permutation as image are obtained by replacing the crossings in the permutation diagrams by positive or negative crossings.  Hence the claim holds. Now since $\beta$ is a square-free braid, any two strands cross at most once, and obviously one cannot continuously deform a positive crossing in a negative one. Hence the number and type of crossing between any two strands is constant on the set of reduced diagrams for $\beta$. It follows that a strand which is above all the other in some reduced diagram must be above all the others in all the reduced diagrams.
\end{rmq}

\begin{defn}\label{def:mikado}
We define \emp{Mikado braids} recursively as 
\begin{enumerate}
\item The braid $e$ is a Mikado braid in $\cB_1$.
\item A braid $\beta\in \mathcal{B}_{n+1}$ is a Mikado braid if in any reduced diagram for $\beta$, there exists at least one good strand and if removing any such strand yields a Mikado braid in $\mathcal{B}_n$.
\end{enumerate}
\end{defn}
We denote by $\Bgnn$ the set of Mikado braids. An example of a Mikado braid in $\mathcal{B}_7$ is given in Figure \ref{fig:mikado}.
\begin{rmq}\label{def:dehornoy}
Notice that Definition \ref{def:mikado} is equivalent to the definition of an \emp{$f$-realizable} braid from \cite[Section 2]{Dehornoy}.
\end{rmq}

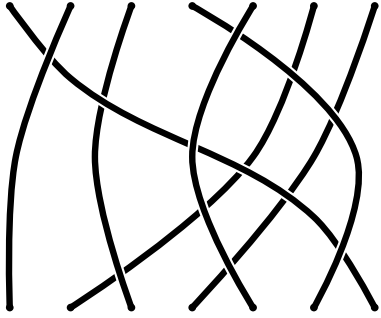
\begin{figure}[h!]
\begin{center}
\psscalebox{0.8}{\begin{pspicture}(0,0)(6,5)
\pscurve[linewidth=3pt](1,0)(4,2.5)(5,5)
\pscurve[linewidth=5pt, linecolor=white](2,0)(1.4,2.5)(2,5)
\pscurve[linewidth=3pt](2,0)(1.4,2.5)(2,5)
\pscurve[linewidth=5pt, linecolor=white](3,0)(5,2.5)(6,5)
\pscurve[linewidth=3pt](3,0)(5,2.5)(6,5)
\pscurve[linewidth=5pt, linecolor=white](0,5)(1,3.8)(5,1.5)(6,0)
\pscurve[linewidth=3pt](0,5)(1,3.8)(5,1.5)(6,0)
\pscurve[linewidth=5pt, linecolor=white](5,0)(5.7,2.5)(3,5)
\pscurve[linewidth=3pt](5,0)(5.7,2.5)(3,5)
\pscurve[linewidth=5pt, linecolor=white](4,0)(3,2.5)(4,5)
\pscurve[linewidth=3pt](4,0)(3,2.5)(4,5)
\pscurve[linewidth=5pt, linecolor=white](1,5)(0.1,2.5)(0,0)
\pscurve[linewidth=3pt](1,5)(0.1,2.5)(0,0)
\psdot(0,5)
\psdot(1,5)
\psdot(2,5)
\psdot(3,5)
\psdot(4,5)
\psdot(5,5)
\psdot(6,5)

\psdot(0,0)
\psdot(1,0)
\psdot(2,0)
\psdot(3,0)
\psdot(4,0)
\psdot(5,0)
\psdot(6,0)

\end{pspicture}}
\end{center}
\caption{A mikado braid in $\mathcal{B}_7$.}
\label{fig:mikado}
\end{figure}

\begin{rmq}\label{rmq:unseulbrin}
If $\beta\in\Bpnn$ with good $i$\ts{th} strand such that removing it yields a
braid $\beta'\in\Bgn$, then $\beta\in\Bgnn$ (\ie, we can assume that the inductive condition in point $(2)$ of Definition \ref{def:mikado} is true for one strand instead of any strand). We prove it by induction on rank: indeed, assume that $\beta$ has another good $j$\ts{th} strand. We must show that $\widetilde{\beta}\in \mathcal{B}_n$ obtained by removing the $j$\ts{th} strand is in $\Bgn$. The $i$\ts{th} and $j$\ts{th} strand do not cross. The $j$\ts{th} strand of $\beta$ may have become the $(j-1)$\ts{th} strand of $\beta'$ but it is still good. In particular, we can remove it, yielding a braid $\beta''\in \mathcal{B}_{n-1}^{\mathrm{Mik}}$. But $\beta''$ is also obtained from $\widetilde{\beta}$ by removing the strand corresponding to the $i$\ts{th} strand in $\beta$ (which may be the $(i-1)$\ts{th} strand of $\widetilde{\beta}$), hence by induction $\widetilde{\beta}\in\Bgn$ since $\beta''\in \mathcal{B}_{n-1}^{\mathrm{Mik}}$ is.
\end{rmq}   

\begin{prop}\label{prop:caract}
Let $\beta\in\mathcal{B}_{n+1}$. Then $\beta$ is a Mikado braid if and only
if $\beta=\bx^{-1}\by$ for some $\bx, \by\in\bW$, hence by Proposition
\ref{prop:gperm} if and only if $\beta$ a rational permutation braid.  
\end{prop}

\begin{proof}
First, assume that $\beta$ is a rational permutation braid. We argue by induction on $n$. The trivial
braid $e$ has the claimed form. Now assume $\beta\in\Bgnn$ and consider any reduced braid
diagram for it. Then there exists at least one good strand. Removing the rightmost such strand we
get a braid $\beta'\in \Bgn$ which by induction can be written as $\bx'^{-1}\by'$ for
$\bx',\by'\in\bW'$, where $W'=\left\langle s_1,\dots, s_{n-1}\right\rangle$. Assuming that the
strand removed joins the $i$\ts{th} point of the sequence above to the $j$\ts{th} point of the
sequence below, one then has that $\beta=\bs_i^{-1}\bs_{i+1}^{-1}\dots \bs_n^{-1} \bx'^{-1} \by'
\bs_n \bs_{n-1}\dots \bs_j$ (see Figure \ref{figure:brindevant}), where $\bx'$ and $\by'$ are seen
in $\mathcal{B}_{n+1}$ under the embedding $\iota_n:\mathcal{B}_n\hookrightarrow
\mathcal{B}_{n+1}$. But for any $k$, the parabolic Coxeter element $s_n s_{n-1}\dots s_k$ is left-reduced with respect to $W'$. In particular for $k=i,j$, it implies that both $\by' \bs_n \bs_{n-1}\dots\bs_j$ and $\bx' \bs_n \bs_{n-1}\dots\bs_i$ lie in $\bW$.

Conversely, assume that $\beta=\bx^{-1}\by$ with $\bx,\by\in\bW$. Note
that in a reduced braid diagram for a (positive) permutation braid, \ie, an
element of $\bW$, for all $i$ the $i$\ts{th} strand is above the stands numbered
$1$ to $i-1$. Hence in $\bx^{-1}$ the strand which ends at
position $n+1$ is above all the others, the one ending at
position $n$ is just below, and so on. It implies that when
concatenating reduced diagrams for
$\bx^{-1}$ and $\by$, in the resulting diagram, the $(n+1)$\ts{th}
strand of $\by$ is above all the others, the $n$\ts{th} strand of $\by$ is
just below, and so on. Hence the Mikado property is
satisfied, so that $\beta$ is a Mikado braid.
\end{proof}
\begin{center}
\begin{figure}
\psscalebox{0.9}{\begin{pspicture}(8,0)(6,3.5)
\psdot(0.5,0)
\psdot(1,0)
\psdot(1.5,0)
\psdot(2,0)
\psdot(2.5,0)
\psdot(3,0)
\psdot(3.5,0)
\psdot(4,0)
\psdot(4.5,0)
\psdot(0.5,3)
\psdot(1,3)
\psdot(1.5,3)
\psdot(2,3)
\psdot(2.5,3)
\psdot(3,3)
\psdot(3.5,3)
\psdot(4,3)
\psdot(4.5,3)
\psline(0.5,0)(0.5,3)
\psline(1,0)(1,3)
\psline(1.5,0)(1.5,1)
\psline(2,0)(2,3)
\psline(2.5,0)(2.5,3)
\psline(3,1)(3,3)
\psline(3.5,0)(3.5,3)
\psline(4,0)(4,3)
\psline(4.5,0)(4.5,3)
\psframe[linecolor=gray, fillstyle=solid, fillcolor=gray](0.25, 2.1)(4.75, 0.6)
\pscurve[linewidth=0.1](3,0)(2.7,0.9)(1.6,2.1)(1.5,3)
\psdot(6.5,0)
\psdot(7,0)
\psdot(7.5,0)
\psdot(8,0)
\psdot(8.5,0)
\psdot(9,0)
\psdot(9.5,0)
\psdot(10,0)
\psdot(10.5,0)
\psdot(6.5,3)
\psdot(7,3)
\psdot(7.5,3)
\psdot(8,3)
\psdot(8.5,3)
\psdot(9,3)
\psdot(9.5,3)
\psdot(10,3)
\psdot(10.5,3)
\psline(6.5,0)(6.5,3)
\psline(7,0)(7,3)
\pscurve(7.5,2.1)(7.55,2.4)(7.9,2.7)(8,3)
\pscurve(8,2.1)(8.05,2.4)(8.4,2.7)(8.5,3)
\pscurve(8.5,2.1)(8.55,2.4)(8.9,2.7)(9,3)
\pscurve(9,2.1)(9.05,2.4)(9.4,2.7)(9.5,3)
\pscurve(9.5,2.1)(9.55,2.4)(9.9,2.7)(10,3)
\pscurve(10,2.1)(10.05,2.4)(10.4,2.7)(10.5,3)
\psline(7.5,0)(7.5,1)
\psline(8,0)(8,1.5)
\psline(8.5,0)(8.5,1.5)
\psline(9,1)(9,1.5)
\pscurve(9.5,0)(9.1,0.4)(9,0.6)
\pscurve(10,0)(9.6,0.4)(9.5,0.6)
\pscurve(10.5,0)(10.1,0.4)(10,0.6)
\psframe[linecolor=gray, fillstyle=solid, fillcolor=gray](6.25, 2.1)(10.25, 0.6)
\pscurve[linecolor=white, linewidth=0.27](7.5,3)(7.7,2.8)(10.4,2.3)(10.5,2.1)
\pscurve[linewidth=0.1](7.5,3)(7.7,2.8)(10.4,2.3)(10.5,2.1)
\pscurve[linecolor=white, linewidth=0.27](9,0)(9.1,0.2)(10.4,0.4)(10.5,0.6)
\pscurve[linewidth=0.1](9,0)(9.1,0.2)(10.4,0.4)(10.5,0.6)
\pscurve[linewidth=0.1](10.5,0.6)(10.6, 1.35)(10.5,2.1)
\rput(8,1.4){\large ${\beta'}$}
\psline[linestyle=dotted](6,0)(13.5,0)
\psline[linestyle=dotted](6,0.6)(13.5,0.6)
\psline[linestyle=dotted](6,2.1)(13.5,2.1)
\psline[linestyle=dotted](6,3)(13.5,3)
\rput(5.5,1.5){$\rightsquigarrow$}
\rput(12,0.2){$\bs_n\bs_{n-1}\cdots \bs_j$}
\rput(12.1,2.5){$\bs_i^{-1}\bs_{i+1}^{-1}\cdots \bs_n^{-1}$}
\rput(11.8,1.35){$\iota(\beta')$}
\end{pspicture}}
\caption{Illustration of the inductive process used in the proof of Proposition \ref{prop:caract}.}
\label{figure:brindevant}
\end{figure}
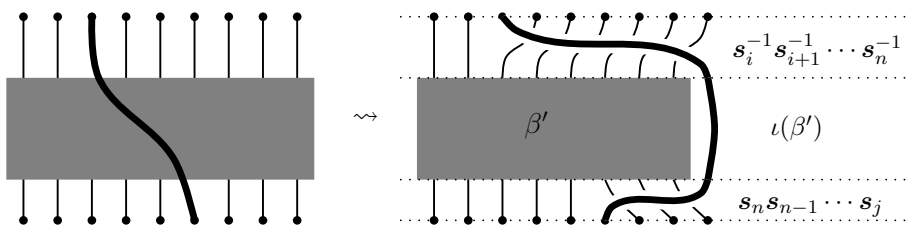
\end{center}

This gives in particular a non-inductive algebraic characterization of
rational permutation braids. To summarize, putting Propositions \ref{prop:gperm}, \ref{prop:caract} and Remark \ref{def:dehornoy} together we get:
\begin{thm}\label{thm:equivA}
Let $\beta\in \mathcal{B}_{n+1}$. The following are equivalent:
\begin{enumerate}
\item The braid $\beta$ is a Mikado braid.
\item The braid $\beta$ is a rational permutation braid.
\item There exist $\bx, \by\in\bW$ such that $\beta=\bx^{-1}\by$.
\item There exist $\bx,\by\in\bW$ such that $\beta=\bx\by^{-1}$.
\item The braid $\beta$ is $f$-realizable in the sense of \cite{Dehornoy}.
\end{enumerate}
\end{thm}

\begin{rmq}\label{enumerationA}
If one is interested in enumerative combinatorics, it is natural to ask for
the number $\mathrm{Mik}(n)$ of Mikado braids in $\mathcal{B}_{n}$. Using
Theorem \ref{thm:equivA}, we see that counting Mikado braids is equivalent to
counting pairs of permutations
$(x,y)\in\mathfrak{S}_{n}\times\mathfrak{S}_{n}$ with no common left
descents. Indeed, there may be distinct pairs $(\bx, \by)$ and $(\bx',
\by')\in\bW\times \bW$ such that $\bx^{-1}\by=\bx'^{-1}\by'$, but taking
$\bx$ and $\by$ with no common left divisor in $B(W)^+$ gives unicity. These
pairs have been counted in \cite{CSV} and correspond to the coefficients of a
power series given by a Bessel function. The first values are
$\mathrm{Mik}(1)=1$, $\mathrm{Mik}(2)=3$, $\mathrm{Mik}(3)=19$,
$\mathrm{Mik}(4)=211$, \dots . There does not seem to exist a simple closed formula for $\mathrm{Mik}(n)$. 
\end{rmq}

\subsection{Simple dual braids are Mikado braids}

The aim of this Section is to show that the images in $\mathcal{B}_{n+1}$ of the simple elements of a dual braid monoid $B_c^*$, where $c$ is a standard Coxeter element, are Mikado braids. To this end we first describe a model for the group of fractions of $B_c^*$ using braids in a cylinder and then explain how the embedding of $B_c^*$ into the Artin braid group $\mathcal{B}_{n+1}$ can be viewed using this model. 

\subsubsection{Noncrossing partitions}\label{sec:coxeterarbitr}

Given a standard Coxeter element $c$, we graphically describe the elements of $\DIV(c)$ as follows. The Coxeter
elements in $W$ are exactly the $(n+1)$-cycles; a Coxeter element $c$ is standard if and only if it
is an $(n+1)$-cycle $c=(i_1, i_2,\ldots,i_k,\ldots i_{n+1})$ with $i_1=1$, $i_k=n+1$, with the property that
the sequence $i_1i_2\cdots i_k$ increasing and the sequence $i_k\cdots i_{n}
i_{n+1}$ decreasing (see \cite[Lemma 8.2]{GobWil}). To represent the
elements of $\DIV(c)$ we use $n+1$ points labeled by $i_1,
i_2,\ldots,i_k,\ldots i_{n+1}$ in clockwise order on a circle. Like in the case where
$c=s_1s_2\cdots s_n$, the elements of $\DIV(c)$ can be represented as a union
of polygons having vertices the marked points as follows: to
each cycle occurring in the decomposition of $x\in\DIV(c)$ into
a product of
disjoint cycles, one associates the polygon obtained as convex hull of
the set of points on the circle labeled by the elements in the
support of the cycle (in particular a polygon can be reduced to an edge or
even a single point). Such a polygon with vertices labelled
$(i_1,i_2,\ldots,i_k)$ in clockwise order corresponds to the cycle
$(i_1,i_2,\ldots,i_k)$ which is the unique cycle dividing $c$ with
set of vertices $\{i_1,\ldots, i_k\}$.
The polygons obtained are pairwise disjoint;
equivalently the partition defined by the cycle decomposition of $x$ is
noncrossing for this choice of labeling of the circle depending on $c$.  To
emphasize that the property to be a noncrossing partition depends on
the labeling we will speak of a noncrossing partition of the sequence
$(i_1, i_2,\ldots,i_k,\ldots i_{n+1})$.
So we have defined a bijection between $\DIV(c)$ and the set of noncrossing
partitions of the sequence $(i_1, i_2,\ldots,i_k,\ldots i_{n+1})$.

It will be convenient for the proofs 
to draw the point with label $1$ at the top of the circle, the
point with label $n+1$ at the bottom, the points with label in $i_k\cdots i_n
i_{n+1}$ on the left and the points with label in $i_1i_2\cdots i_k$ on the
right, each point having a specific height depending on its label: that is,
if $P$, $Q$ are two points with respective labels $i, j\in\{1,\dots, n+1\}$,
$i<j$, then $P$ is higher than $Q$. An example is given in Figure \ref{figure:orientation}. In case we represent a noncrossing partition we may use curvilinear polygons instead or regular polygons for a more comfortable reading (see Figure \ref{fig:noncarb} on the left).  

\begin{figure}[h!]
\begin{center}
\begin{tabular}{ccc}
& \begin{pspicture}(1,0)(6,3)
\pscircle(3,1.5){1.5}
\psdots(4.5,1.5)(3.75,2.79)(3.75,.21)(2.25,.21)(2.25,2.79)(1.5,1.5)
\psdots[dotsize=4.5pt](3.75,.21)(2.25,.21)
\psdots[dotsize=4.5pt](1.5,1.5)(2.25,2.79)(4.5,1.5)(3.75,2.79)
\rput(4.05,2.92){\textrm{{\footnotesize \textbf{3}}}}
\rput(4.8,1.5){\textrm{{\footnotesize \textbf{4}}}}
\rput(4,.03){\textrm{{\footnotesize \textbf{6}}}}
\rput(2,.03){\textrm{{\footnotesize \textbf{5}}}}
\rput(1.17,1.5){\textrm{{\footnotesize \textbf{2}}}}
\rput(1.88,2.92){\textrm{{\footnotesize \textbf{1}}}}
\end{pspicture} & \begin{pspicture}(0,0)(6,3)
\pscircle(3,1.5){1.5}

\psline[linestyle=dotted, linewidth=0.4pt](1,3)(5,3)
\psline[linestyle=dotted, linewidth=0.4pt](1,2.4)(5,2.4)
\psline[linestyle=dotted, linewidth=0.4pt](1,1.8)(5,1.8)
\psline[linestyle=dotted, linewidth=0.4pt](1,1.2)(5,1.2)
\psline[linestyle=dotted, linewidth=0.4pt](1,0.6)(5,0.6)
\psline[linestyle=dotted, linewidth=0.4pt](1,0)(5,0)

\psdots(3,3)(1.8, 2.4)(4.4696, 1.8)(4.4696, 1.2)(1.8, 0.6)(3,0)
\rput(5.2,3){\small $1$}
\rput(5.2,2.4){\small $2$}
\rput(5.2,1.8){\small $3$}
\rput(5.2,1.2){\small $4$}
\rput(5.2,0.6){\small $5$}
\rput(5.2,0){\small $6$}
\end{pspicture}
\end{tabular}
\end{center}
\caption{Example of a labeling of the vertices given by the standard Coxeter element $c=s_2 s_1 s_3 s_5 s_4=(1,3,4,6,5,2)$.}
\label{figure:orientation}
\end{figure}
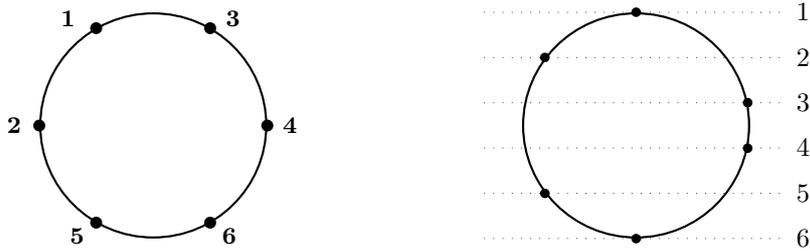

\subsubsection{Dual braids and graphical representation}\label{sec:grapha}

We follow in this part \cite{BDM}, in particular {\it loc.\ cit.} Section 1
and Corollary 3.5, and refer to it for the results.

\begin{definition}
A \emp{dual braid} is an element of the group of fractions of $B_c^*$. A \emp{simple dual braid} is a dual braid which is a simple element of $B_c^*\subset\mathrm{Frac}(B_c^*)$ as described in Section \ref{sub:dbm}.
\end{definition}

To describe the simple dual braids we consider braids where the starting points and the ending points of the strands are on two parallel circles. The points of the two circles are labeled as in Section \ref{sec:coxeterarbitr}. More precisely, we consider one unit circle in a plane with one complex coordinate at level $t=0$ and a
second unit circle in another such plane at level $t=1$. 
So the starting points
of the strands have coordinates $(z_k,0)_{k=1,\ldots,n+1}$ with $|z_k|=1$,
and the ending points have coordinates $ (z_k,1)$.  The Coxeter element is $c=(i_1,\dots, i_{n+1})$ as in Section \ref{sec:coxeterarbitr}, when along the circle in clockwise order starting from $(z_{i_1}, 0)$ the startpoints of the strands appear in the order $(z_{i_1},0)$, $(z_{i_2},0)$, \dots, $(z_{i_{n+1}},0)$, and the same for the endpoints, replacing the second coordinate by $1$.
  
For each pair $\{i,j\}$ one lifts the reflection $(i,j)$ to
the braid $\delta_{i,j}$ where the strands starting
with $(z_k,0)$ for $k\neq i,j$ has fixed first coordinate and the strands
starting with $z_i$ (resp.\ $z_j$) is represented by the path $t\mapsto
(t,(z_i+z_j)/2+(z_i-z_j)/2(\cos(\pi t)+i\varepsilon\sin(\pi t))$ (resp.\ same
with $i$ and $j$ exchanged)  where $\varepsilon$ is small enough (see Figure \ref{tresse (i,j)}).
We draw the figures with the level $t=0$ above and the level $t=1$ below. 

\begin{figure}[h!]
\resizebox{40mm}{60mm}{\input{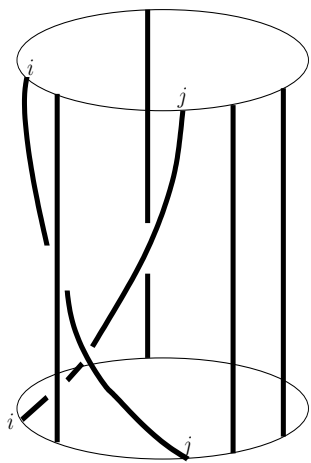}}
\caption{Braid diagram for the dual braid $\delta_{i,j}$.}\label{tresse (i,j)}
\end{figure}

These braids statisfy the relations in the dual presentation of Proposition \ref{presentation of B*} for the braid
group $\cB_{n+1}$, that is $\delta_{i,j}\delta_{j,k}=\delta_{j,k}\delta_{k,i}$
when $z_i,z_j,z_k$ are in clockwise order and $\delta_{i,j}\delta_{k,l}=\delta_{k,l}\delta_{i,j}$ when 
$i,j,k,l$ are 4 distinct points with $[i,j]$ and $[k,l]$ noncrossing (these are exactely the cases where the
product of two reflections divide the Coxeter element).
The simple dual braids are in one-to-one correspondence with the 
noncrossing partitions of the sequence $(i_1,\ldots,i_{n+1})$ as described in Section \ref{sec:coxeterarbitr}. The braid diagram associated to a simple dual braid represented by a polygon
$(z_{j_1},\ldots, z_{j_k})$ where the points are in clockwise ordering is the
product $\delta_{j_1,j_2}\delta_{j_2,j_3}\ldots\delta_{j_{k-1},j_k}$. In the
resulting braid diagram the strand starting with $(z_{j_m},0)$ ends with $(z_{j_{m-1}},1)$ for $1<m\leq k$ while the strand starting with $(z_{j_1}, 0)$ ends with $(z_{j_k}, 1)$.
(see Figure \ref{dualsimple}).
\begin{figure}[h!]
\resizebox{40mm}{60mm}{\input{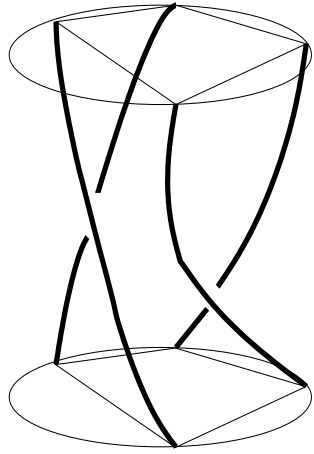}}
\caption{The braid associated to a quadrangle.}\label{dualsimple}
\end{figure}

If the Coxeter element is standard the dual braids $\delta_{i,i+1}$ satisfy the usual
braid relations and the identification of the group of fractions of the dual braid
monoid with the braid group $\cB_{n+1}$ maps $\delta_{i,i+1}$ to $\sigma_i$.
Graphically to recover the usual braid diagram of the image in $\mathcal{B}_{n+1}$ of a simple dual braid $\beta$, one has to put the points $z_{i_1},\ldots z_{i_{n+1}}$ on the
circle in such a way that the imaginary parts are in decreasing order and
``look from the right'' to the braid. This gives a diagram for an Artin braid, which is a horizontal mirror diagram of a diagram for the image in $\mathcal{B}_{n+1}$ of $\beta$. Equivalently starting from the noncrossing partition representation of $\beta$, one orders the polygons in counterclockwise order and then projects everything to the right; this gives the braid seen from the bottom (see Figure \ref{fig:noncarb}). 

The point in this description of the embedding which will be crucial later and follows from our graphical descriptions of the embedding above is that if a polygon in the noncrossing partition representation of $\beta$ has no polygon at its right (we assume that an edge and a single point are polygons), then the strands from this polygon are above all the strands coming from other polygons in the classical Artin braid representation of the image of $\beta$, as one can see in Figure \ref{fig:noncarb} with the polygon reduced to the edge labeled by $4$: on the very right, the strand labeled by $4$ appears above all the other strands. 

\begin{remark}
Note that the point where the Coxeter element needs to be standard is the
identification of $\delta_{i,i+1}$ with $\sigma_i$. If the Coxeter element is
not standard the braids $\delta_{i,i+1}$ do not satisfy the braid
relations.
\end{remark}

\begin{figure}[h!]
\psscalebox{1.08}{
\begin{pspicture}(1,-0.5)(10.9,3)
\pscircle(3,1.5){1.5}
\pscurve[linecolor=gray, linewidth=1.5pt](3,3)(2.5, 2.5)(1.8, 2.4)
\pscurve[linecolor=gray, linewidth=1.5pt](3,3)(3.55, 2.1)(4.4696, 1.8)
\pscurve[linecolor=gray, linewidth=1.5pt](3,0)(2.6,1.2)(1.8, 2.4)
\pscurve[linecolor=gray, linewidth=1.5pt](3,0)(3.5,0.9)(4.4696, 1.8)

\psline[linestyle=dotted, linewidth=0.4pt](1,3)(5,3)
\psline[linestyle=dotted, linewidth=0.4pt](1,2.4)(5,2.4)
\psline[linestyle=dotted, linewidth=0.4pt](1,1.8)(5,1.8)
\psline[linestyle=dotted, linewidth=0.4pt](1,1.2)(5,1.2)
\psline[linestyle=dotted, linewidth=0.4pt](1,0.6)(5,0.6)
\psline[linestyle=dotted, linewidth=0.4pt](1,0)(5,0)

\psline[linestyle=dotted, linewidth=0.4pt](5.4,3)(7.5,3)
\psline[linestyle=dotted, linewidth=0.4pt](5.4,2.4)(7.5,2.4)
\psline[linestyle=dotted, linewidth=0.4pt](5.4,1.8)(7.5,1.8)
\psline[linestyle=dotted, linewidth=0.4pt](5.4,1.2)(7.5,1.2)
\psline[linestyle=dotted, linewidth=0.4pt](5.4,0.6)(7.5,0.6)
\psline[linestyle=dotted, linewidth=0.4pt](5.4,0)(7.5,0)

\psdots(3,3)(1.8, 2.4)(4.4696, 1.8)(4.4696, 1.2)(1.8, 0.6)(3,0)
\rput(5.2,3){\small $1$}
\rput(5.2,2.4){\small $2$}
\rput(5.2,1.8){\small $3$}
\rput(5.2,1.2){\small $4$}
\rput(5.2,0.6){\small $5$}
\rput(5.2,0){\small $6$}

\pscurve[linecolor=gray, linewidth=1.5pt]{<-}(6.7,3)(7.2,2.4)(6.7,1.8)
\pscurve[linecolor=gray, linewidth=1.5pt]{->}(6.7,3)(6.2,2.7)(6.7,2.4)
\pscurve[linecolor=gray, linewidth=1.5pt]{<-}(6.7,0)(7.2,0.4)(5.7,1.4)(6.7,2.4)
\pscurve[linecolor=gray, linewidth=1.5pt]{->}(6.7,0)(7.1,-0.1)(7.4,0.5)(6.2,1.4)(6.7,1.8)

\psdots(6.7,0)(6.7,0.6)(6.7,1.2)(6.7,1.8)(6.7,2.4)(6.7,3)

\psline[linewidth=1.5pt](10,0)(10,3)

\pscurve[linecolor=white, linewidth=3pt](8.8,3)(9.2,2)(10.2,1.1)(10.4,0)
\pscurve[linecolor=gray, linewidth=1.5pt]{->}(8.8,3)(9.2,2)(10.2,1.1)(10.4,0)

\pscurve[linecolor=white, linewidth=3pt]{->}(10.4,3)(10.2,2)(9.4,1)(9.2,0)
\pscurve[linecolor=gray, linewidth=1.5pt]{->}(10.4,3)(10.2,2)(9.4,1)(9.2,0)

\psline[linecolor=white, linewidth=3pt](9.6,0)(9.6,3)
\psline[linewidth=1.5pt](9.6,0)(9.6,3)

\pscurve[linecolor=white, linewidth=3pt]{->}(8.4,3)(8.45,1.8)(8.7,1)(8.8,0)
\pscurve[linecolor=gray, linewidth=1.5pt]{->}(8.4,3)(8.45,1.8)(8.7,1)(8.8,0)

\pscurve[linecolor=white, linewidth=3pt]{->}(9.2,3)(9,2)(8.6,1)(8.4,0)
\pscurve[linecolor=gray, linewidth=1.5pt]{->}(9.2,3)(9,2)(8.6,1)(8.4,0)

\psdots(8.4,0)(8.8,0)(9.2,0)(9.6,0)(10,0)(10.4,0)
\psdots(8.4,3)(8.8,3)(9.2,3)(9.6,3)(10,3)(10.4,3)

\rput(8.4,-0.4){\footnotesize $1$}
\rput(8.8,-0.4){\footnotesize $2$}
\rput(9.2,-0.4){\footnotesize $3$}
\rput(9.6,-0.4){\footnotesize $4$}
\rput(10,-0.4){\footnotesize $5$}
\rput(10.4,-0.4){\footnotesize $6$}

\end{pspicture}}
\caption{Passing from a dual to a classical braid.}
\label{fig:noncarb} 
\end{figure}
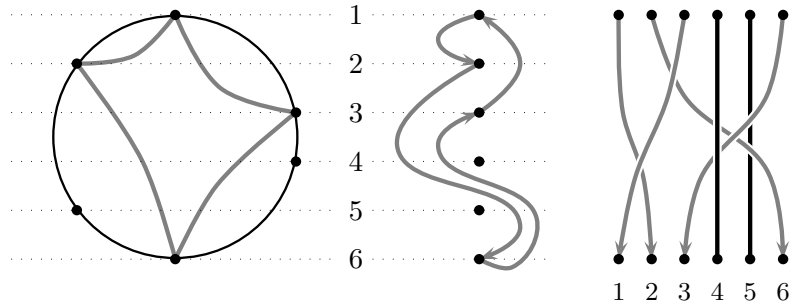

\subsubsection{Removing the strands of simple dual braids}
The aim here is to use the graphical representation of simple elements from
the previous section to show that simple elements of dual braid monoids are
rational permutation braids. The argumentation is therefore topological.

Any $\bu\in\bD_c$ always has at least one polygon. The polygons of the graphical representation correspond to the cycles (or blocks) of the corresponding permutation (or noncrossing partition). Consider any block represented by either a polygon or a point which we denote by $P$ and which is right to any other polygon or point, that is, with no other block at its right in the picture (with our conventions described above). It always exists, since the polygons are disjoint and with vertices on the circle. One can then inductively remove all the strands corresponding to edges of $P$ one after the other such that the removed strand at each step is good, starting with any strand having no other strand from $P$ at its right (it always exists and one can keep going inductively since $P$ is a polygon): such a strand is necessarily good. After having removed it one can find a strand which is good at this step, and so on.

After having removed all the strands from $P$, we go on with another polygon $Q$ which is right to all the remaining polygons, and so on, until all the strands have been removed. Hence we showed:

\begin{thm}\label{simplesmikadoa}
Let $c$ be a Coxeter element in a Coxeter group $(W,S)$ of type $A_n$. Any element of
$\bD_c\subset B_c^*\subset B(W)$ is a Mikado braid, and thus a rational permutation braid.
\end{thm}

\subsubsection{Additional properties in case the Coxeter element is linear}
We denote by $<$ the Bruhat order on $W$. In the next proposition, we show that in case $c=s_1s_2\cdots s_n$, one can always remove a strand in a braid representing $\bu\in\bD_c$ such that the resulting braid $\bv$ lies in $\bD_{c'}$ for $c'=s_1s_2\cdots s_{n-1}$; this allows us to say a bit more in that case.

\begin{prop}
Let $c=s_1s_2\cdots s_n$, $\bu\in\bD_c$. Then $\bu\neq e$ can be written in the form $\bx^{-1}\by$ with $\bx,\by\in\bW$ and $p(\bx)< p(\by)$.
\end{prop}
\begin{proof}
We argue by induction on the rank. We prove the statement with the inequality $<$ replaced by $\leq$. It is then clear that the inequality is an equality if and only if $\bu=e$. The result is trivially true for $n=1$.  Assume $n>1$. With our choice of Coxeter element, all indices lie on the right part of the circle. If there is a block of $u$ reduced to a single index $i$, then the $i$\ts{th} strand of $\bu$ must be unbraided, and moreover since all the points lie on the right of the circle, the $i$\ts{th} strand lies over all the other strands. We then argue as in the proof of Proposition \ref{prop:caract}; that is, one has $\bu=\bs_i^{-1}\bs_{i+1}^{-1}\dots \bs_n^{-1} \bu' \bs_n \bs_{n-1}\dots \bs_i$ where $\bu'\in \mathcal{B}_n$ is the braid obtained from $\bu$ by removing the $i$\ts{th} strand.  But $\bu'$ is then in $\bD_{c'}$, where $c'=s_1\cdots s_{n-1}$: indeed, it is graphically obtained from the noncrossing representation of $\bu$ by removing the point with label $i$ and subtracting $1$ from each label larger than $i$, still yielding a diagram of a simple element $\bu'$ for the smaller linear Coxeter element $c'$. By induction $\bu'=\bx'^{-1}\by'$ with $p(\bx')<p(\by')$. Since $s_n\cdots s_i$ is left-reduced with respect to the parabolic subgroup $\left\langle s_1,\dots, s_{n-1}\right\rangle$ we get that $p(\bx'\bs_n\cdots \bs_i)<p(\by' \bs_n\cdots \bs_i)$, hence we have the claimed property with $\bx=\bx'\bs_n\cdots \bs_i$, $\by=\by'\bs_n\cdots \bs_i$. 

Hence we can assume that there is no unbraided strand. But in that case, since the Coxeter element is linear, there must be a polygon in the graphical representation of $\bu$ having an edge $(i-1,i)$, that is an $i$\ts{th} strand ending at $i-1$ which is over all the other strands that it crosses.  We argue as above, writing this time $\bu=\bs_i^{-1}\bs_{i+1}^{-1}\dots \bs_n^{-1} \bu' \bs_n \bs_{n-1}\dots \bs_{i-1}$. The element $\bu'$ is the simple element for $c'$ obtained by contracting the edge $(i-1,i)$, hence also identifying the points with labels $i$ and $i-1$ and subtracting $1$ from any point with label bigger than $i$. Arguing as above we get the claim.  
\end{proof}

\section{Mikado braids of type $B_n$}\label{sec:b}

\subsection{Coxeter and Artin-Tits groups of type $B_n$}

The Coxeter group $W$ of type $B_n$ is the group of fixed points in the
Coxeter group $W'$ of type $A_{2n-1}$ under the diagram automorphism $\tau$ which
exchanges $s_i$ and $s_{2n-i}$. This fact lifts to the corresponding Artin-Tits
groups (see \eg, \cite[Corollary $4.4$]{Jean}). So we shall see  the Artin-Tits
group $B(W)$ of type $B_n$ as the group of fixed points in the Artin-Tits group of type $A_{2n-1}$ under the diagram automorphism  which we still denote by $\tau$,  that is, $B(W')^{\tau}=B(W)$.
Topologically this means that the Artin-Tits group $B(W)$ can be seen as
the group of symmetric braids in $\cB_{2n}$ (the symmetry exchanges the
$i$\ts{th} and the $(2n+1-i)$\ts{th} strands and exchanges under-/over-crossings). 

We write $\mathcal{B}_{2n}^B$ for the braids of type $B_n$ identified with 
$\mathcal{B}_{2n}^\tau$. For convenience, since $\mathcal{B}_{2n}^B$ consists of symmetric braids, we label the
strands of $\cB_{2n}$ by $-n, -n+1,\dots, -1,1,\dots, n$ rather than by $1,2,\dots, 2n$. We write $t_0,\dots, t_{n-1}$ for the Coxeter generators of $W$ with relations

\begin{center}
$\begin{array}{lcll}
t_0t_1t_0t_1&=&t_1t_0t_1t_0, & \\
t_i t_{i+1} t_i&=&t_{i+1}t_i t_{i+1} & \text{if }1\leq i\leq n-2,\\
t_i t_j&=&t_j t_i & \text{if }|i-j|>1.
\end{array}$
\end{center}
Inside $W'$ one has  $t_0=s_n$ and $t_i=s_is_{2n-i}$ for $i=1,\ldots,
n-1$, where 
$s_1=(-n,-n+1)$, $s_2=(-n+1,-n+2)$, \dots, $s_n=(-1,1)$, $s_{n+1}=(1,2)$,
\dots, $s_{2n-1}=(n-1,n)$  in  $\mathfrak{S}_{2n}$ seen as the group of
permutations of $\{-n,\ldots,-1,1,\ldots,n\}$.

This lifts to the braid groups: the generators of $B(W)$ seen in
$\cB_{2n}$ are $\bt_0=\bs_n$ and $\bt_i=\bs_i\bs_{2n-i}$ for
$i=1,\ldots,n-1$.

\subsection{Mikado braids of type $B_n$}\label{sec:mikbrb}

\begin{defn}
A braid $\beta\in \mathcal{B}_{2n}^B$ is a \emp{Mikado braid} (of type
$B_n$) if when viewed in $\cB_{2n}$, starting from any diagram $D$ for
$\beta$ one can inductively remove pairs of symmetric strands, one being
above all the other strands (so that its symmetric is under all the
other strands).
\end{defn}
We write $\mb$ for the set of Mikado braids of type $B_n$. It follows from
the definition that
$\mb=\cB_{2n}^B\cap\mathcal{B}_{2n}^{\mathrm{Mik}}$. Hence we can use the geometric properties of Mikado braids of type $A_n$ given in Section \ref{sec:mikbr}.

\begin{prop}\label{prop:caractb}
Let $\beta\in\mathcal{B}_{2n}^B$; then $\beta$ is a Mikado braid of type
$B_n$ if and only if $\beta=\bx^{-1}\by$ for some $\bx, \by\in\bW$, hence by
Proposition \ref{prop:gperm} if and only if $\beta$ a rational permutation braid.  
\end{prop}
Note that here $\bW$ refers to the lift of the Coxeter group $W$ of type
$B_n$ in the braid group of type $B_n$.

\begin{proof}

A braid $\beta\in\cB_{2n}$ is in $\mb$ if and only if it is a Mikado braid in
$\mathcal{B}_{2n}$ and it is fixed by $\tau$. By Proposition \ref{prop:caract}, this is equivalent to 
$\beta$ being of the form $\bx\inv \by$ with $\bx$ and $\by$ in $\bW'$ and
$\beta$ fixed by $\tau$. Since there are gcds in Artin-Tits monoids of spherical
type, the expression $\bx\inv\by$ is unique under the condition that $\bx$ and
$\by$ have no common left-divisor (in the Artin-Tits monoid of type $A_{2n-1}$).
Since $\tau$ induces an automorphism of the monoid, if $\bx$ and $\by$ satisfy the unicity
condition,  then $\bx\inv \by$ is $\tau$-fixed if and only if each one of
$\bx$ and $\by$ is $\tau$-fixed, that is, is in $\bW^{\prime\tau}=\bW$.
\end{proof}

This gives in particular a non-inductive algebraic characterization of
rational permutation braids of type $B_n$. To summarize, putting Propositions \ref{prop:gperm} and \ref{prop:caractb} together we get:
\begin{thm}\label{thm:caractb}
Let $\beta\in B(W)\cong \mathcal{B}_{2n}^B$. The following are equivalent:
\begin{enumerate}
\item The braid $\beta$ is a Mikado braid.
\item The braid $\beta$ is a rational permutation braid.
\item There exist $\bx, \by\in\bW$ such that $\beta=\bx^{-1}\by$.
\item There exist $\bx,\by\in\bW$ such that $\beta=\bx\by^{-1}$.
\end{enumerate}
\end{thm}

\begin{problem}\label{enumerationB}
What is the number of Mikado braids of type $B_n$? Unlike in type $A_n$ (see Remark \ref{enumerationA}), we have no hint towards an answer to this question. 
\end{problem}

\subsection{Simple dual braids are Mikado braids}

\subsubsection{Graphical representation of simple elements}

Through the embbeding $W=W^{\prime\tau}\subset W'$, the standard Coxeter
elements of $W$ identify with the $\tau$-fixed standard Coxeter elements of $W'$. Moreover if $T'$ denotes
the set of reflections of $W'$ and $T$ the set of reflections of $W$, and if $c$
is a $\tau$-fixed standard Coxeter element of $W'$, then for $x$
and $y$ in $W$ one has $x\ldiv_T y\ldiv_T c$ if and only if $x\ldiv_{T'}
y\ldiv_{T'} c$ (see \cite[Lemma 4.8]{kapi}).
\begin{remark}
Note that a reflection of $W$ is not a reflection of $W'$ in general.
\end{remark}

Since the dual braid monoid of type $A_{2n-1}$ is defined by a presentation with generators the left-divisors of the chosen Coxeter
element $c$ and with relations given by $\ell_{T'}$-shortest decompositions
of $c$, the above observations imply that the dual braid monoid of type $B_n$ identifies with the fixed points 
of $\tau$ in the dual braid monoid
of type $A_{2n-1}$ for the same  Coxeter element.

The $\tau$-fixed standard Coxeter elements of $W'$ are the $2n$-cycles of the
form $c=(i_1, i_2,\dots, i_n, -i_1, -i_2,\dots, -i_n)$ with $\{i_1,i_2,\dots, i_n\}=\{1,2,\dots, n\}$ and the sequence $i_1i_2\cdots i_n$ first increasing, then decreasing. 

We also get that the elements of $\DIV(c)$ for a Coxeter element  $c\in W$ are in
one-to-one correspondence with the $\tau$-fixed noncrossing partitions of 
$2n$ points labelled $-n,\ldots,-1,1,2,\ldots,n$ on a circle with the
clockwise order on the labels in the order given by $c$ (see Figure \ref{figure:orientationb}).  

These symmetric partitions are called \textit{noncrossing partitions of type $B_n$} associated to $c$ (see \cite{Reiner}).

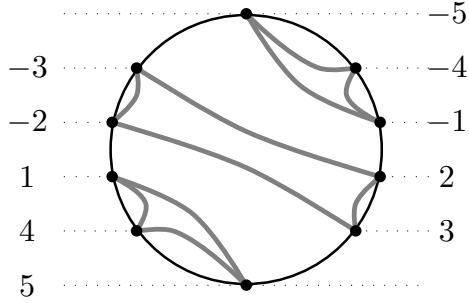
\begin{figure}[h!]
\begin{center}
\psscalebox{1.2}{
\begin{pspicture}(0,0)(6,3)
\pscircle(3,1.5){1.5}

\pscurve[linecolor=gray, linewidth=1.5pt](3,3)(3.6, 2.2)(4.4696, 1.8)
\pscurve[linecolor=gray, linewidth=1.5pt](3,3)(3.8,2.4)(4.2,2.4)
\pscurve[linecolor=gray, linewidth=1.5pt](4.4696, 1.8)(4.1,2.1)(4.2,2.4)

\pscurve[linecolor=gray, linewidth=1.5pt](3,0)(2.4, 0.8)(1.5304,1.2)
\pscurve[linecolor=gray, linewidth=1.5pt](3,0)(2.2,0.6)(1.8,0.6)
\pscurve[linecolor=gray, linewidth=1.5pt](1.8,0.6)(1.9,0.9)(1.5304,1.2)

\pscurve[linecolor=gray, linewidth=1.5pt](4.4696, 1.2)(3,1.7)(1.8,2.4)
\pscurve[linecolor=gray, linewidth=1.5pt](4.2, 0.6)(3,1.3)(1.5304,1.8)
\pscurve[linecolor=gray, linewidth=1.5pt](4.2, 0.6)(4.2,0.9)(4.4696, 1.2)
\pscurve[linecolor=gray, linewidth=1.5pt](1.5304,1.8)(1.8,2.1)(1.8,2.4)

\psline[linestyle=dotted, linewidth=0.4pt](1,3)(5,3)
\psline[linestyle=dotted, linewidth=0.4pt](1,2.4)(1.8,2.4)
\psline[linestyle=dotted, linewidth=0.4pt](4.2,2.4)(5,2.4)
\psline[linestyle=dotted, linewidth=0.4pt](1,1.8)(1.5304,1.8)
\psline[linestyle=dotted, linewidth=0.4pt](4.4696,1.8)(5,1.8)
\psline[linestyle=dotted, linewidth=0.4pt](1,1.2)(1.5304,1.2)
\psline[linestyle=dotted, linewidth=0.4pt](4.4696,1.2)(5,1.2)
\psline[linestyle=dotted, linewidth=0.4pt](1,0.6)(1.8,0.6)
\psline[linestyle=dotted, linewidth=0.4pt](4.2,0.6)(5,0.6)
\psline[linestyle=dotted, linewidth=0.4pt](1,0)(5,0)

\psdots(3,3)(1.8, 2.4)(4.2,2.4)(4.4696, 1.8)(1.5304,1.8)(4.4696, 1.2)(1.5304,1.2)(1.8, 0.6)(4.2,0.6)(3,0)
\rput(5.2,3){\small $-5$}
\rput(5.2,2.4){\small $-4$}
\rput(5.2,1.8){\small $-1$}
\rput(5.2,1.2){\small $2$}
\rput(5.2,0.6){\small $3$}
\rput(0.6,0){\small $5$}

\rput(0.6,2.4){\small $-3$}
\rput(0.6,1.8){\small $-2$}
\rput(0.6,1.2){\small $1$}
\rput(0.6,0.6){\small $4$}
\end{pspicture}}

\end{center}
\caption{Example of a labeling of the vertices given by the Coxeter element
$c=t_1 t_2 t_0 t_4 t_3$ and of a noncrossing partition of type $B_5$ for this Coxeter element. The
corresponding element of $\DIV(c)$ is $(-5,-4,-1)(5,4,1)(2,3,-2,-3)$.}
\label{figure:orientationb}
\end{figure}

Such a partition corresponds to a $\tau$-fixed braid in $\cB_{2n}$ by the
same process as explained in Subsection \ref{sec:grapha}.

\subsubsection{Removing pairs of strands of simple dual braids}
The argumentation to show inductively that any simple dual braid $\bx\in\bD_c$ of type $B_n$ is a Mikado braid can be led as in type $A_n$, that is, by looking at the graphical representation of $\bx$, equivalently the number of blocks of the noncrossing partition $x\in\DIV(c)$, and removing pairs of strands. We start from a polygon $P$ which is to the right of any other polygon. Either $P$ has a symmetric polygon $\bar{P}$ such that $j$ indexes a vertex of $P$ if and only if $-j$ indexes a vertex of $\bar{P}$; or $P$ is its own symmetric, in particular $P$ has vertices on both sides of the circle and an index $j$ indexes a vertex of $P$ if and only if $-j$ indexes a vertex of $P$. 

In the first case, we remove inductively pairs of strands of $P$, one corresponding to an edge of $P$ and its symmetric which is an edge of $\bar{P}$ as we did in type $A_n$, in such an order that at each step, the removed strand from $P$ is above all the others (which implies that the corresponding symmetric strand which is also removed is below all the others), that is, the corresponding edge should be right to any other edge of $P$. In the second case, we do exactly the same, except that pairs of strands of $P$ are removed simultaneously. 

After having removed all the strands from $P$, we go on with another polygon $Q$ which is right to all the remaining polygons, and so on, until all the strands have been removed. Hence we showed:

\begin{thm}\label{simplesmikadob}
Let $c$ be a Coxeter element in a Coxeter group $(W,S)$ of type $B_n$. Any element of
$\bD_c\subset B_c^*\subset B(W)$ is a Mikado braid of type $B_n$, and thus a
rational permutation braid.
\end{thm}

\section{Dihedral and exceptional types}\label{sec:exceptional}

\begin{theorem}\label{simplesmikadoexc}
Let $c$ be a Coxeter element in a Coxeter system $(W,S)$ of type
$H_3$, $H_4$, $I_2(m)$ ($m\geq 3$), $E_n$ ($n=6,7,8$) or $F_4$; any element of
$\bD_c\subset B_c^*\subset B(W)$ is a rational permutation braid.
\end{theorem}
\begin{proof}
For exceptional types we have checked the result using the program CHEVIE
(\cite{chevie}). Let us prove it for dihedral types.
By Proposition \ref{prop:gperm}., it suffices to show that any element in
$\bD_c$ has the form $\bx\by^{-1}$ for $\bx$, $\by\in\bW$. Since $S=\{s,t\}$
there are only two possible choices of Coxeter element. Let $c=st$. Then the
elements of $\bD_c$ are the identity element, $\bc$ and the elements of $\bT_c$. By considering the Hurwitz action on $(\bs,\bt)$ we see that the elements of $\bT_c$ are exactly the $(\underbrace{\bs\bt\bs\cdots}_{k})(\underbrace{\bs\bt\bs\cdots}_{k-1})^{-1}$ with $1\leq k\leq m$, which concludes the proof. 
\end{proof}

\section{Positivity properties}

The whole paper is motivated by a positivity conjecture on the
expansion of the images of simple dual braids in a canonical basis of the
Iwahori-Hecke algebra of the Coxeter group. Using work from the previous sections we
prove these properties here for finite irreducible Coxeter groups of type
other than type $D_n$.

\subsection{Iwahori-Hecke algebra of a Coxeter system} 

\begin{defn}
Let $(W,S)$ be a Coxeter system. For $s,t\in S$, recall that $m_{s,t}$ denotes the order
of $st$ in $W$. The \emp{Iwahori-Hecke algebra} $H(W)=H(W,S)$ of $(W,S)$ is the associative, unital $\mathbb{Z}[v,v^{-1}]$-algebra generated by a copy $\{T_s\mid s\in S\}$ of $S$ with relations
\begin{center}
$\begin{array}{ll}
T_s^2=(v^{-2}-1) T_s+ v^{-2} & \forall s\in S,\\
\underbrace{T_s T_t\cdots}_{m_{s,t}~\text{copies}}=\underbrace{T_t T_s\cdots}_{m_{s,t}~\text{copies}}, & \forall s,t\in S.
\end{array}$
\end{center}
\end{defn}

The algebra $H(W)$ has a standard basis $\{T_w\}_{w\in W}$ where $T_w$ is the image of $\bw\in\bW$ under the unique group morphism $a:B(W)\rightarrow H(W)^{\times}$ such that $\bs\mapsto T_s$ for all $s\in S$. It also has two canonical bases $\{C_w\}_{w\in W}$ and $\{C'_w\}_{w\in W}$ defined in \cite{KL}. The algebra $H(W)$ has a unique semilinear involution $j_{H}$ such that $j_H(T_{s})=-v^2 T_{s}$ for all $s\in S$ and $j_H(v)=v^{-1}$. The two canonical bases are then related by the equalities 
\begin{equation}\label{eq:cc'}
C_w=(-1)^{\ls(w)} j_H(C'_w),~\forall w\in W.
\end{equation}
A reference for this material is \cite[Section 7.9]{Humph}.
\subsection{Positivity properties of simple dual braids}

There are many positivity statements involving the canonical bases of $H(W)$.
In this section we prove positivity results on the expansion of images of
simple dual braids in $H(W)$ when they are expressed in the basis $\{C_w\}$.
This comes as a corollary of the results of the previous sections and of the following:

\begin{thm}[{\cite[Corollary 2.9]{DL}}]\label{thm:dl}
Let $(W,S)$ be a finite Weyl group, then for all $x,y\in W$, one has $$T_x^{-1} T_y\in\sum_{w\in W}\mathbb{N}[v,v^{-1}] C_w.$$
\end{thm}
\begin{rmq}
Dyer and Lehrer showed that in case $(W,S)$ is a finite Coxeter group, the
statement of the above theorem (for fixed $x,y$) is equivalent to the
statement that $C'_{x^{-1}} T_y$ has a positive expansion in the standard
basis; they then show that this last statement holds for Weyl groups
(\cite[Proposition 2.7 and Theorem 2.8]{DL}). Dyer later showed in \cite[Conjecture 7(b)]{D}\footnote{The authors thank Matthew Dyer for pointing out
this fact to the second author.}
(using partially unpublished results) that the Kazhdan-Lusztig positivity
conjecture (now proven in \cite{EW} as a corollary of Soergel conjecture \cite{S}) implies the positivity of the expansion of
$C'_{x^{-1}} T_y$ in the standard basis for all finite Coxeter groups. Hence we can assume that Theorem \ref{thm:dl} holds for all finite Coxeter groups.
\end{rmq}
Theorem \ref{thm:dl} together with Proposition \ref{prop:gperm} yields:
\begin{prop}\label{prop:gpermpositive}
Let $\beta\in B(W)$ be a rational permutation braid. Then $$a(\beta)\in\sum_{w\in W} \mathbb{N}[v,v^{-1}] C_w.$$
\end{prop}
\begin{proof}
This is an immediate consequence of Theorem \ref{thm:dl}, Proposition \ref{prop:gperm} and the fact that $a(\bw)=T_w$ for any $\bw\in \bW$.
\end{proof}

\begin{thm}\label{thm:positivity}
Let $(W,S)$ be a finite irreducible Coxeter system of type other than $D_n$ and
let $c$ be any Coxeter element in $W$; then for any $\bu\in\bD_c$, one has that $$a(\bu)\in\sum_{w\in W} \mathbb{N}[v,v^{-1}] C_w.$$
\end{thm}
\begin{proof}
This follows from Theorems \ref{simplesmikadoa}, \ref{simplesmikadob},
\ref{simplesmikadoexc} and Proposition \ref{prop:gpermpositive}.
\end{proof}
Notice that the proof of this theorem required to prove that simple dual
braids can be written in the form $\bx^{-1} \by$, which was done in types $A_n$
and $B_n$ by using the geometry of Artin braids.

We are not able to prove this property for type $D_n$, but both computations and the fact that it holds in any other spherical type leads us to conjecture that it again holds:
\begin{conjecture}\label{conj:d}
Let $(W,S)$ be a Coxeter system of type $D_n$.
Let $c$ be any Coxeter element in $W$; then
any element of
$\bD_c\subset B_c^*\subset B(W)$ is a rational permutation braid.
\end{conjecture}
Note that if this conjecture is true we will have the same positivity property for type
$D_n$ as for all other spherical types.

\begin{problem}
Does there exist a model for the braid group of type $D_n$ by Artin-like
braids which could be used to prove that simple elements of dual braid
monoids are rational permutation braids? A model for type $D_n$ can be found in \cite{allcock}, but the analogues of the Mikado braids in this model are not the expected ones. 
\end{problem}

\begin{problem}
Is there a uniform approach to show that simple elements of dual braid
monoids are rational permutation braids?
\end{problem}

\subsection{Temperley-Lieb algebra and monomial basis}\label{sec:tl}
Since elements of $\bW$, that is, images of the simple elements for the classical Garside structure provide a basis of the Iwahori-Hecke algebra, one might investigate the linear independence of images of elements of $\bD_c$ in $H(W)$. The set $\bD_c$ is too small to give a basis of $H(W)$: indeed, the algebra $H(W)$ has rank $|W|$ while $\DIV(c)$ is in general a strict subset of $W$ counted by the generalized Catalan number of type $W$ (see for instance~\cite{Armst}). Nevertheless, in type $A_n$, the set $\bD_c$ gives a basis of a remarkable quotient of $H(W)$, the \emp{Temperley-Lieb algebra}. From now on, $(W,S)$ is a Coxeter system of type $A_n$.
\begin{defn}
The \emp{Temperley-Lieb algebra}
$\mathrm{TL}_n$ is the associative, unital $\mathbb{Z}[v, v^{-1}]$-algebra
obtained as quotient of $H(W)$ by the two-sided ideal generated by the
elements
$$\sum_{w\in\left\langle s_i,s_{i+1}
\right\rangle} T_w,$$ for $i=1,\ldots,n-1$, where
 $s_i=(i,i+1)$. We write
$\theta:H(W)\twoheadrightarrow \mathrm{TL}_n$ for the quotient map.
Alternatively, one gets an isomorphic algebra by taking the quotient
by the two-sided ideal generated by the
elements $$\sum_{w\in\left\langle s_i,s_{i+1}
\right\rangle}(-1)^{\ell_{\mathcal{S}}(w)} v^{2\ell_{\mathcal{S}}(w)} T_w,$$
for $i=1,\dots, n-1$
(see \eg, \cite[Section 2.3 and Remark 2.4]{Ram}). We write
$\theta':H(W)\twoheadrightarrow \tl$ for this alternative quotient map. 
\end{defn}

The algebra $\tl$ has a presentation by generators $b_{s_1},\dots, b_{s_n}$ and relations
\begin{center}
$\begin{array}{rcll}
b_{s_i} b_{s_{i\pm 1}} b_{s_i} &=&b_{s_i} & \forall i, i\pm 1\in\{1,\dots, n\},\\
b_{s_i}b_{s_j}&=&b_{s_j}b_{s_i} & \forall i,j\in\{1,\dots, n\}\text{ with }|i-j|>1,\\
b_{s_i}^2&=&(v+v^{-1}) b_{s_i} & \forall i\in\{1,\dots, n\}.
 \end{array}$

\end{center}
We have (see \cite[Section 2.3 and Remark 2.4]{Ram})
\begin{eqnarray}\label{quotients}
\theta(T_{s_i})=v^{-1}b_{s_i}-1,~\theta'(T_{s_i})=v^{-2}-v^{-1} b_{s_i}.
\end{eqnarray}
\begin{defn}
An element $w\in W$ is \emp{fully commutative} if one can pass from any
reduced $S$-decomposition of $w$ to any other one by applying only relations of the form $s_i s_j=s_j s_i$ for $|i-j|>1$. We denote by $W_f$ the set of fully commutative elements. For details on fully commutative elements we refer to \cite{Stem}.
\end{defn}
The algebra $\tl$ has a basis indexed by the elements of $W_f$. It
is built as follows; given any $w\in W_f$ with reduced expression
$st\cdots u$, the above presentation shows that the product $b_s b_t\cdots b_u$ does not
depend on the choice of the reduced expression and therefore we can denote it by
$b_w$. The set $\{b_w\}_{w\in W_f}$ turns out to yield a basis of $\tl$ (see \cite[Corollary 5.32]{KaTu}).

\subsection{Projection of the canonical basis}
The basis $\{b_w\}_{w\in W_f}$ turns out to be related to the canonical basis $\{C'_w\}_{w\in W}$ of $H(W)$ as follows:
\begin{thm}[\cite{FG}, Section $3$]\label{thm:fg}
The basis $\{b_w\}_{w\in W_f}$ is the projection of $\{C'_w\}_{w\in W}$ by $\theta$. In symbols, for $w\in W_f$, one has $\theta(C'_w)=b_w$ while $\theta(C'_w)=0$ for $w\notin W_f$.
\end{thm}

\subsection{Zinno basis}
The Temperley-Lieb algebra comes equipped with a semilinear involution $j_{\tl}$ such that $j_{\tl}(v)=v^{-1}$ and $j_{\tl}(b_s)=b_s$ for any $s\in S$. To realize the algebra $\tl$ as a quotient of $\mathbb{Z}[v, v^{-1}][\mathcal{B}_n]$, we use the composition of $a':B(W)\rightarrow H(W)$, $\bs\mapsto v T_s$ with $\theta'$, that is, by~\eqref{quotients} any generator $\bs_i\in\mathcal{B}_n$ is mapped to $v^{-1}-b_{s_i}$. We set $\omega:=\theta'\circ a'$. 

\begin{thm}[\cite{Z}, Theorem $2$, \cite{LL}, Theorem $1$]
Let $c=s_n s_{n-1}\cdots s_1$; then $\omega(\bD_c)$ yields
a basis of $\tl$.
\end{thm}
The above Theorem is stated differently in \cite{Z} since dual braid monoids
had not been introduced when that paper appeared. Notice that the
conventions used in \cite{Z} correspond to the quotient map $\theta\circ a$,
which sends $\bs_i$ to $v^{-1}b_{s_i}-1$. Since one passes from one to the other by
composing with a semilinear automorphism of $\mathbb{Z}[v,
v^{-1}]$-modules and we are here interested in bases, this does not affect the result.

This result is also generalized to arbitrary standard Coxeter elements in \cite[Corollary 5.2.9]{Vincenti} and \cite[Theorem 3.8.28 and Remark 3.8.29]{GobTh}:

\begin{thm}\label{thm:basetl}
For any standard Coxeter element $c$, the set $\omega(\bD_c)$ yields a basis of $\tl$.
\end{thm}
Hence the images of the simple elements of the Garside monoids $B_c^*$ yield a basis of $\tl$, in much the same way that the images of the simple elements of the Garside monoid $B^+(W)$ yield the standard basis of $H(W)$.

Note that one may have $x\in\DIV(c)\cap \DIV(c')$ for some $c'\neq c$, in which case the corresponding
elements of $\bD_c$ and $\bD_{c'}$, hence also the corresponding basis elements of $\tl$ are not equal in general. For example, in type $A_2$ with $c=s_1 s_2$, $c'=s_2 s_1$, the reflection $s_1 s_2 s_1$ lies in $\DIV(c)\cap \DIV(c')$ (as any reflection). The corresponding element of $\bD_c$, \resp $\bD_{c'}$ is $\bs_1 \bs_2\bs_1^{-1}$, \resp $\bs_1^{-1} \bs_2 \bs_1$. The two elements are not equal.  

\subsection{Positivity consequences}
The two quotient maps from Section \ref{sec:tl} are related by $\theta=j_{\tl}\circ\theta'\circ j_H$. Using this fact together with Theorems \ref{thm:positivity}, \ref{thm:fg} and Equality \eqref{eq:cc'} we get:
\begin{thm}\label{thm:positivitetl}
Let $c$ be any standard Coxeter element, let $\{Z_x\}_{x\in\DIV(c)}$ be the
corresponding basis of $\tl$, \ie, $Z_x:=\omega(\bx)$ for $x\in\DIV(c)$. Then for any $x\in \DIV(c)$, one has $$Z_x\in\sum_{w\in W_f} \mathbb{N}[v,v^{-1}] (-1)^{\ls(w)} b_w.$$
\end{thm} 
One can show that there exist total orders on the sets $\DIV(c)$ and $W_f$ such that the base
change from $\{b_w\}_{w\in W_f}$ to $\{Z_x\}_{x\in\DIV(c)}$ is upper triangular. Zinno shows it for
$c=s_1s_2\cdots s_n$ in \cite{Z} and the result is extended in the second author's thesis \cite{GobTh} to
arbitrary $c$. As noticed in \cite{Gob}, for $c=s_1s_2\cdots s_n$, the order given by Zinno can
be refined into the Bruhat order on $\DIV(c)$, which is studied extensively and combinatorially in \cite{GobWil}, where the orders to consider for other choices of Coxeter elements are also described. Closed formulas for some of the coefficients (in the inverse matrix) are given in \cite{Gob} in case the Coxeter element is $c=s_1s_2\cdots s_n$ (which can be obviously adapted for $c=s_n s_{n-1}\cdots s_1$), but we do not have a closed formula for them in full generality.

\begin{remark}
In type $B_n$, Theorem \ref{thm:basetl} also holds (see \cite{Vincenti2}). In that case, the diagram basis of the Temperley-Lieb quotient is not the projection of the canonical basis of the corresponding Iwahori-Hecke algebra (see \cite{FG}), hence there is no analogue of Theorem \ref{thm:positivitetl} in that case. In type $D_n$, the cardinality of $|\bD_c|$ is bigger than the dimension of the Temperley-Lieb algebra, but the images of the elements of $\bD_c$ in it still generate the whole algebra, as shown in \cite{Vincenti}. 
\end{remark}


\begin{thebibliography}{4}

\bibitem{allcock} D.~Allcock, \textsl{Braid pictures for Artin groups}, Trans.\ Amer.\ Math.\ Soc.\
{\bf 354} (2002), 345--3474.

\bibitem{Armst} D.~Armstrong, \emph{Generalized noncrossing partitions and combinatorics of Coxeter groups}, Mem. Amer. Math. Soc. {\bf 202} (2009), no. 949.

\bibitem{BDSW} B.~Baumeister, M.~Dyer, C.~Stump, and P.~Wegener, \textsl{A note on the transitive Hurwitz action on
decompositions of parabolic Coxeter elements}, Proc.\ Amer.\
Math.\ Soc., Ser.\ B, {\bf 1} (2014), 149--154.

\bibitem{Dual} D.~Bessis, \textsl{The dual braid monoid}, Ann.\
Sci.\ \'Ecole Normale Sup\'erieure {\bf 36} (2003), 647--683.

\bibitem{BDM} D.~Bessis, F.~Digne, and J.~Michel, \textsl{Springer theory
in braid groups and the Birman-Ko-Lee monoid}, Pacific J.\ Math.\ {\bf 205} (2002),
287--309. 

\bibitem{BKL} J.~Birman, K.H.~Ko, and S.J.~Lee,  \textsl{A New Approach to
the Word and Conjugacy Problems in the Braid Groups}, Adv.\ in Math.\ 
{\bf 139} (1998), 322--353.

\bibitem{bourbaki} N.~Bourbaki, \textsl{Groupes et alg\`ebres de Lie, chapitres 4,5
et 6}, Masson (1981).

\bibitem{BS} E.~Brieskorn and K.~Saito, \textsl{Artin-Gruppen und Coxeter-Gruppen}, Invent.\ Math.\
{\bf 17} (1972), 245--271.

\bibitem{muhl} N.~Brady, J.~McCammond, B.~M{\"u}hlherr, and W.~Neumann,
\textsl{Rigidity of Coxeter groups and Artin groups},
Geom.\ Dedicata {\bf 94} (2002), 91--109. 

\bibitem{kapi} T.~Brady and C.~Watt, \textsl{$K(\pi, 1)$'s for Artin groups of finite type}, Geom.\
Dedicata {\bf 94} (2002), 225--230.

\bibitem{BW} T.~Brady and C.~Watt, \textsl{Noncrossing partitions lattices in finite real reflection
groups}, Trans.\ Amer.\ Math.\ Soc.\ {\bf 360} (2008), 1983--2005.

\bibitem{broue-michel} M.~Brou\'e and J.~Michel, \textsl{Sur certains
\'el\'ements r\'eguliers des groupes de Weyl et les vari\'et\'es de
Deligne-Lusztig associ\'ees}, Prog.\ Math.\ {\bf 141} (1997), 73--139.

\bibitem{CSV} L.~Carlitz, R.~Scoville, and T.~Vaughan, \textsl{Enumeration of pairs of permutations
and sequences}, Bull.\ Amer.\ Math.\ Soc.\ {\bf 80} (1974), 881--884.

\bibitem{Dehornoy} P.~Dehornoy, \textsl{Three-dimensional realizations of
braids}, J.\ London Math.\ Soc.\ {\bf 60} (1999), 108--132.

\bibitem{DDGKM}
P.~Dehornoy, F.~Digne, D.~Krammer, E.~Godelle, and J.~Michel.
\textsl{Foundations of Garside theory}, Tracts in Mathematics {\bf 22}, Europ.\ Math.\ Soc.\
(2015).

\bibitem{DP} P.~Dehornoy and L.~Paris, \textsl{Gaussian groups and Garside
groups, two generalizations of Artin groups}, Proc.\ London Math.\ Soc.\ {\bf 79}
(1999), 569--604.

\bibitem{Del} P.~Deligne, \textsl{Les immeubles des groupes de tresses généralisés}, 
Invent.\ Math.\ {\bf 17} (1972), 273--302.

\bibitem{D} M.J.~Dyer, \textsl{Representation theories from Coxeter groups},
Representations of groups (Banff, AB, 1994),  CMS Conf.\ Proc.\ {\bf 16}, (1995), 105--139
Amer.\ Math.\ Soc., Providence, RI. 

\bibitem{Dyernil} M.J.~Dyer, \textsl{Modules for the dual nil Hecke ring}, \url{http://www3.nd.edu/~dyer/papers/nilhecke.pdf}.

\bibitem{DL} M.J.~Dyer and G.I.~Lehrer, \textsl{On positivity in Hecke
algebras}, Geom.\ Dedicata {\bf 35} (1990), 115--125. 

\bibitem{EW} B.~Elias and G.~Williamson, \textsl{The Hodge theory of Soergel
bimodules}, Ann.\ of Math.\ {\bf 180} (2014), 1089--1136.

\bibitem{FG} C.K.~Fan and  R.M.~Green, \textsl{Monomials and Temperley-Lieb
algebras}, J.\ Algebra {\bf 190} (1997), 498--517. 

\bibitem{GP} M.~Geck and G.~Pfeiffer, \textsl{Characters of finite Coxeter groups and Iwahori-Hecke
algebras}, London Math. Soc. Monographs, New Series {\bf 21}, Oxford University Press (2000).

\bibitem{FHM} W.N.~Franzsen, R.B.~Howlett, and B.~M{\"u}hlherr \textsl{Reflections in abstracts
Coxeter groups} Comment. Math. Helv. {\bf 81} (2006) 665--697 

\bibitem{GobTh} T.~Gobet, \textsl{Bases of Temperley-Lieb algebras}, PhD
thesis, Universit\'e de Picardie, (2014), \url{http://www.mathematik.uni-kl.de/fileadmin/AGs/agag/gobet/files/these.pdf}.

\bibitem{Gob} T.~Gobet, \textsl{Noncrossing partitions, fully commutative elements and bases of the
Temperley-Lieb algebra}, accepted in Journal of Knot Theory and its Ramifications, \url{http://arxiv.org/abs/1409.6500}. 

\bibitem{GobWil} T.~Gobet and N.~Williams, \textsl{Noncrossing partitions and Bruhat order}, European Journal of Combinatorics {\bf 53} (2016), 8-34.

\bibitem{Ram} T.~Halverson, M.~Mazzocco and A.~Ram, \textsl{Commuting
families in Hecke and Temperley-Lieb algebras}, Nagoya Math.\ J.\ {\bf 195} (2009), 125--152. 

\bibitem{Humph} J.~Humphreys, \textsl{Reflection groups and Coxeter groups}, Cambridge Studies in
Advanced Mathematics {\bf 29}, Cambridge University Press (1990).

\bibitem{KaTu} C.~Kassel, V.~Turaev, \textsl{Braid groups}, Graduate Texts in Mathematics {\bf 247} (2008), Springer, New York.

\bibitem{KL} D.~Kazhdan and G.~Lusztig, \textsl{Representations of Coxeter
Groups and Hecke Algebras}, Invent.\ Math.\ {\bf 53} (1979), 165--184. 

\bibitem{LL} E.K.~Lee and S.J.~Lee, \textsl{Dual presentation and linear
basis of the Temperley-Lieb algebras}, J.\ Korean Math.\ Soc.\ {\bf 47} (2010), 445--454.

\bibitem{Jean} J.~Michel, \textsl{A note on words in braid monoids}, J.\ 
Algebra {\bf 215} (1999), 366--377.

\bibitem{chevie} J.~Michel, \textsl{The development version of the CHEVIE
package of GAP3}, J.\ Algebra {\bf 435} (2015), 308--336.

\bibitem{Reiner} V.~Reiner, \textsl{Non-crossing partitions for classical
reflection groups}, Discrete Math.\ {\bf 177} (1997), 195--222.

\bibitem{S} W.~Soergel, \textsl{Kazhdan-Lusztig polynomials and indecomposable bimodules over
polynomial rings}, J.\ Inst.\ Math.\ Jussieu {\bf 6} (2007), 501--525.

\bibitem{Stem} J.R.~Stembridge, \textsl{On the fully commutative elements of Coxeter groups} 
J.\ of Algebraic Comb.\ {\bf 5} (1996), 353--385. 

\bibitem{Vincenti2} C.~Vincenti, \textsl{Algèbre de Temperley-Lieb de type B}, C.~R.~Math.\ Acad.\
Sci.\ Paris {\bf 342} (2006), 233--236.

\bibitem{Vincenti} C.~Vincenti, \textsl{Monoïde dual, antichaînes de racines et algèbres de
Temperley-Lieb}, PhD thesis, Universit\'e de Picardie, (2007).

\bibitem{Z} M.G.~Zinno, \textsl{A Temperley-Lieb basis coming from the
braid group}, J.\ Knot Theory Ramifications {\bf 11} (2002), 575--599.
\end{thebibliography}
\end{document}